\documentclass[10pt]{amsart}
\usepackage{amssymb,latexsym,amscd,amsmath,amsthm,manfnt}
\usepackage{mathdots}
\usepackage[matrix,arrow]{xy}
\input{amssym.def}
\input{xy}

\xyoption{all}

\theoremstyle{definition}

\numberwithin{equation}{subsection} 
\newtheorem{guess}{theorem}[subsection]

\newtheorem{rem}[guess]{Remark}

\newtheorem{thm}[guess]{Theorem}

\newtheorem{prop}[guess]{Proposition}
\newtheorem{Cor}[guess]{Corollary}

\newcommand{\cO}{\mathcal{O}}

\newcommand{\cG}{\mathcal{G}}

\newcommand{\cF}{\mathcal{F}}

\newcommand{\Pic}{\mathrm{Pic}}

\newcommand{\Spec}{\mathrm{Spec}}

\newcommand{\rank}{\mathrm{rank}}

\newcommand{\lra}{\longrightarrow}

\newcommand{\hra}{\hookrightarrow}

\newcommand{\ra}{\rightarrow}
\newcommand{\ol}{\overline}

\newcommand{\ms}{\mapsto}

\newcommand{\ul}{\underline}

\newcommand{\FF}{\mathbb{F}}

\newcommand{\PP}{\mathbb{P}}
\newcommand{\ZZ}{\mathbb{Z}}
\newcommand{\GG}{\mathbb{G}}
\newcommand{\QQ}{\mathbb{Q}}
\newcommand{\NN}{\mathbb{N}}

\newcommand{\XX}{\mathbb{X}}
\newcommand{\CC}{\mathbb{C}}

\newcommand{\Hom}{\mathrm{Hom}}

\begin{document}

\title{Brauer group of punctual Quot scheme of points on a smooth projective surface}

\author[A. J. Parameshwaran]{A. J. Parameshwaran}
\address{Tata Institute of Fundamental Research, Mumbai }
\email{param@math.tifr.res.in}

\author[Y. Pandey]{Yashonidhi Pandey}
\thanks{The support of Science and Engineering Research Board under Mathematical Research Impact Centric Support File number: MTR/2017/000229 is gratefully acknowledged.}
\address{ 
Indian Institute of Science Education and Research, Mohali Knowledge city, Sector 81, SAS Nagar, Manauli PO 140306, India}
\email{ ypandey@iisermohali.ac.in, yashonidhipandey@yahoo.co.uk}

\begin{abstract} Let $X$ be a smooth projective surface over an algebraically closed field $k$ such  that  $char(k) \neq 2$. Let $X^{[d]}$ denote the punctual Hilbert scheme of zero dimensional quotients of degree $d$ and $X^{(d)}$ denote the symmetric product of $X$. For $\ell \neq 2$, we give a formula for the $\ell$-primary part of the Brauer group of $X^{[2]}$. We show that the Hilbert to Chow morphism induces an isomorphism of cohomological Brauer groups for $d=2$ and a similar result for $d \geq 3$. Let $Q(r,d)$ denote the punctual Quot-scheme parametrising zero dimensional quotients of $\cO_X^{ \oplus r}$ of degree $d$. We show that the natural morphism from $Q(r,d) \ra X^{[d]}$ induces an isomorphism on cohomological Brauer groups. 

\end{abstract}
\maketitle

%\tableofcontents
\section{Introduction}
Let $k$ be an algebraically closed field of characteristic $p \neq 2$ and
 $X$ be a smooth projective surface over $k$. We denote by $X^{[d]}$  the punctual Hilbert scheme of zero dimensional quotients of degree $d$  and by $X^{(d)}$ the $d$-th symmetric product of $X$. Let $C_s(X)$ denote the group of symmetric divisorial correspondences (cf \cite[page 678]{fogarty}).  Fogarty \cite{fogarty} showed that the Picard group of $X^{[d]}$ is naturally isomorphic to $\Pic(X) \times C_s(X) \times \ZZ$. 
Our main result extends this to the description of the Brauer group using invariants of the surface $X$. We denote the Brauer group by $Br(-)$ and the cohomological Brauer group by $Br'(-)$.

In this paper, we assume that $\ell$ is a prime and $\ell \neq p$. For an abelian group $A$, by $A_{\ell^*}$ we mean the set of all elements of $A$ whose order is a pure power of $\ell$. Let $\omega_2: X^{[2]} \ra X^{(2)}$ denote the Hilbert to Chow morphism (\ref{hilbtochow}). We show in Theorem \ref{anychar} that $\omega_2$ induces an isomorphism $Br'(X^{(2)})_{\ell^*} \ra Br'(X^{[2]})_{\ell^*}$. 

Let $\ell$ be odd. Let $[X^{\times 2}/S_2]$ denote the stack quotient of the product of $X$ with itself, $X^{\times 2}$, by the symmetric group $S_2$.  In \S \ref{cohbrgpsym2}, we show that the natural morphisms \begin{eqnarray}
Br'([X^{\times 2}/S_2])_{\ell^*} \ra Br'(X^{\times 2})^{S_2}_{\ell^*} \\
Br'(X^{(2)})_{\ell^*} \ra Br'([X^{\times 2}/S_2])_{\ell^*}
\end{eqnarray}
are isomorphisms.
Let $NS(-)$ denote the N\'eron-Severi group. Set  \begin{equation} 
a:= b_2+b_1^2- rank NS(X)- rank NS(Alb(X)) 
\end{equation} 
where $b_i$ are the $i$-th Betti numbers.
We have 
\begin{eqnarray*} Br(X^{[2]})_{\ell^*}=  (\QQ_{\ell}/\ZZ_{\ell})^{\oplus a} \oplus  H^3(X, \ZZ_{\ell})_{\ell^*} \oplus [H^1(X,\ZZ_{\ell}) \otimes H^2(X, \ZZ_{\ell})]_{\ell^*} \\ \oplus  Tor_1^{\ZZ_{\ell}}(H^1(X,\ZZ_{\ell}),H^3(X,\ZZ_{\ell}))_{\ell^*} \oplus Tor_1^{\ZZ_{\ell}}(H^2(X,\ZZ_{\ell}),H^2(X,\ZZ_{\ell}))_{\ell^*}
\end{eqnarray*} 
 We refer the reader to \S \ref{cohbrgpsym2} for the case $\ell=2$.

Over complex numbers, a sequence (cf Prop \ref{sch}) due to S.Schroer (cf \cite[Prop 1.1]{schroer}) expresses the analytic Brauer group of any complex analytic space as an extension of a finite torsion group by a divisible group. In the setting of complex numbers, methods more elementary than \'etale cohomology are also available. So it is natural to search for a proof using these. We give another proof in Theorem \ref{ansd=2} proving that $\omega_2$ induces an isomorphism of cohomological Brauer groups by showing that it does so separately on the torsion part and the divisible group part. Under the hyptothesis $H_1(X)_{tor}$ has no element of two torsion, we show
\begin{equation}
Br(X^{[2]})=(\QQ/\ZZ)^{\oplus a} \oplus H_1(X)_{tor}^{\oplus (b_1+1)} \oplus Tor_1^{\ZZ}(H^2(X,\ZZ),H^2(X,\ZZ))
\end{equation}

As a corollary we get that the Brauer group  of  $X^{[2]}$ for Godeaux and Catanese surfaces is $\ZZ/5\ZZ^{\oplus 2}$, it is  $(\QQ/\ZZ)^a$ where $a=22-NS(X)$ and $22-2 NS(X)$ for $K3$ surfaces and abelian surfaces respectively and it is trivial for the case of $\PP^2$, Hirzebruch and del Pezzo surfaces (cf Table \ref{table}). 

More generally, over the complex numbers, we can show that
\begin{equation}
Br(X^{[2]})=Br'(X^{(2)})=H_2(X^{(2)},\ZZ)_{tor} \oplus \QQ/\ZZ^{\oplus a}
\end{equation}
In \cite{milgram}, Milgram using results of A.Dold has expressed the homology of the symmetric product of a CW-complex $X$ in terms of the homology of $X$. In this sense, we describe the Brauer group of $X^{[2]}$ reducing it to $X^{(2)}$ which in principle can be computed using Milgram's result (cf Remark \ref{computation}). In Remarks \ref{cctotaro} and \ref{ccsteenrod}, for the case $k=\CC$ we cross-check our results with some consequenes of results of B.Totaro and N.Steenrod respectively.
 
%Using some results of Burt Totaro \cite{totaro}, we show  that when $X$ has torsion-free cohomology  then the Brauer group of  $X^{[2]}$ is $(\QQ/\ZZ)^{a}$. In particular, this result holds whenever $H_1(X)=0$. 

Now we come to the  case of more than two points. In \S $3$, we replace $X^{(2)}$ and $\PP(\Omega_X)$ by more technical invariants. They all have the property that they are built out of the surface $X$. They are summarized in the diagram below where all squares are cartesian
\begin{equation} 
\xymatrix{
X^{[2]} \ar[d] & \PP(\Omega_X) \ar[d] \ar@{_{(}->}[l] & V_2 \ar[d] \ar@{.>}[l] \ar[r] \ar@/_1pc/[ll] & X^{[d]}_* \ar[d] \\
X^{(2)} & X \ar@{_{(}->}[l] & W_2 \ar[l] \ar[r] & X^{(d)}_*
}
\end{equation}
 Roughly speaking, the $d$-th symmetric product $X^{(d)}$ admits a natural stratification and $X^{(d)}_*$ is the union of the two largest stratas while $X^{[d]}_*$ is its inverse image in $X^{[d]}$ under the Hilbert-Chow morphism $\omega_d: X^{[d]} \ra X^{(d)}$.  We show in Theorem \ref{geq3} that $Br'(X^{(d)}_*)_{\ell^*} \ra Br'(X^{[d]}_*)_{\ell^*}$ is an isomorphism for all $\ell \neq p$.

Now let us mention an application of our results. Let $Q(r,d)$ denote the punctual Grothendieck Quot-scheme parametrizing zero dimensional quotients of degree $d$ of $\cO_X^{\oplus r}$.  Given a quotient $q: \cO_X^{\oplus r} \ra S$ in $Q(r,d)$, by considering $\ker(q)  \hra \cO_X^{\oplus r}$,  we can associate the quotient of $\wedge^r ker(q) \hra \cO_X$.  Thus we get a natural morphism
\begin{equation} \label{quothilb}
\phi: Q(r,d) \ra X^{[d]}.
\end{equation} 
In the last section, we show that $\phi$ induces an isomorphism on cohomological Brauer groups. The method of proof is inspired from \cite[Remark 6.3]{bidhhu}.

%\subsection{Conventions and Notations}

%\subsection{Acknowledgement} 

%Let $\cO$ denote the sheaf of holomorphic functions on $X^{[d]}$. Consider the exponential sequence on \begin{equation}0 \ra \ZZ \ra \cO \ra \cO^* \ra 1\end{equation}on $X^{[d]}$. We get the morphism $c: Pic(X^{[d]}) \ra H^2(X^{[d]},\ZZ)$ whose image is the N\'eron-Severi group. 

\section{Hilbert-Chow morphism induces isomorphism on cohomological Brauer group for the case of two points}
\subsection{Generalities on Brauer group of a space}
Let $Z$ be any scheme.  Let $P_1 \ra Z$ and $P_2 \ra Z$ be principal $PGL(r_1)$ and $PGL(r_2)$ bundles on $Z$ respectively. Consider the homomorphism
\begin{equation}
PGL(r_1) \times PGL(r_2) \ra PGL(r_1 r_2).
\end{equation}
Let $P_1 \otimes P_2$ denote the principal $PGL(r_1r_2)$ bundle obtained by means of associated constructions via the above homomorphism and the fibered product $P_1 \times_Z P_2$.

For a principal $GL(m)$ bundle $Q$, let $\PP(Q)$ denote the associated principal $PGL(m)$ bundle.
We say that two principal $PGL$ bundles  $P_1$ and $P_2$ are Brauer equivalent if there exist principal $GL$ bundles $Q_1$ and $Q_2$ such that $P_1 \otimes \PP(Q_1) \simeq P_2 \otimes \PP(Q_2)$ as principal $PGL$ bundles. The Brauer group $Br(Z)$ of a scheme $Z$ consists of equivalences classes of projective bundles \cite{biswasholla}. The trivial class consists of $PGL$ bundles induced by $GL$ bundles. The  group law is defined by tensoring as above. The inverse operation is defined by 
taking the dual projective bundle i.e the projective bundle whose transition functions  are inverse duals of the original bundle.

The cohomological Brauer group is denoted $Br'(Z)$. It is the torsion subgroup of $H^2_{\acute{e}t}(Z,\GG_m)$.  The Brauer group of a smooth quasi-projective variety coincides with the cohomological Brauer group by a theorem of O.Gabber \cite[de Jong]{dj}.  Therefore $Br(Z)=Br'(Z)$ for a smooth quasi-projective variety.

\subsection{Hilbert-Chow morphism} We assume that our base field $k$ is an algebraically closed field of arbitrary characteristic. 
Let $X^{[d]}$ denote the Hilbert scheme of torsion quotients of  $\cO_X$ degree $d$ on $X$. Let $X^{(d)}$ denote the symmetric product obtained by taking quotient of $X^{\times d}$ by the action of the symmetric group $S_d$. Let 
\begin{equation} \label{hilbtochow}
\omega_d: X^{[d]} \ra X^{(d)}
\end{equation}
denote the Hilbert to Chow  morphism. It sends a subscheme $W= \sum_{ p \in Supp(W)} \cO_{W,p}$  of $X$ to $\sum_{p \in Supp(Z)} l(\cO_{W,p}) p $ where $l(\cO_{W,p})$ denotes the length.

We now specialize to the case $d=2$. In this case the singular locus of $X^{(2)}$ is $X$
which itself is smooth.  Let $\Omega_X$ denote the cotangent bundle on $X$. In this paper, by projective bundle of a vector bundle $V$ on $Z$ we shall mean the space of rank one locally free quotients of $V$.

The natural closed diagonal embedding of $i: X \ra X^{(2)}$ gives rise to the following cartesian square over $k$ (\cite[Lemma 4.4]{fogarty})
\begin{equation} \label{diag2}
\xymatrix{
\PP(\Omega_X) \ar[r]_{j} \ar[d]^{\pi} & X^{[2]} \ar[d]^{\omega_2} \\
X \ar[r]^{i} & X^{(2)}
}
\end{equation}
Over an arbitrary algebraically closed field, by \cite[Fogarty page 668 Lemma 4.3]{fogarty} the reduced fibers of $\omega_2$ identify with $\PP^1$ over $X$. It has been shown recently that the fibers of $\omega_2$ are indeed reduced
\cite[Paul-Sebastian \S 3.3 and Prop 3.3.3]{ps}.
We set $U:=X^{(2)} \setminus X$ and $j: U \ra X^{(2)}$ the inclusion.

\begin{prop} \label{pushforwardbyomega2} The natural adjunction maps $\GG_m \ra \omega_{2,*} \omega_2^* \GG_m$ and $\mu_n \ra \omega_{2,*} \omega_2^* \mu_n$ are isomorphisms of \'etale sheaves. 
%Consider the natural morphism $\pi_2: X^{\times 2} \ra X^{(2)}$. The natural adjunction map $\GG_m \ra \pi_{2,*} \pi_2^* \GG_m$ is an isomorphism of \'etale sheaves on $X^{(2)}$.
\end{prop}
\begin{proof} Now $\omega_2^* \GG_m$ (resp $\omega_2^* \mu_n$) identifies with the sheaf defined by $\GG_m$ (resp. $\mu_n$) on $X^{[2]}$. Here below after using this indentification, we get the following commutative diagram where all vertical arrows are obtained from adjunction $Id \ra \omega_{2,*} \omega_2^*$ and the horizontal by $Id \ra j_* j^*$:
\begin{equation} \label{decompositiondiag}
\xymatrix{
i^* \GG_m \ar[r] \ar[d] & i^* j_* j^* \GG_m \ar[d] \\
i^* \omega_{2,*} \GG_m \ar[r] & i^* j_* j^* \omega_{2,*} \GG_m
} \quad
\xymatrix{
i^* \mu_n \ar[r] \ar[d] & i^* j_* j^* \mu_n \ar[d] \\
i^* \omega_{2,*} \mu_n \ar[r] & i^* j_* j^* \omega_{2,*} \mu_n 
}
\end{equation}

Now the left vertical arrows are an isomorphism because  $i^* \omega_{2,*} \GG_m$ (resp. $i^* \omega_{2,*} \mu_n$) identifies with $\GG_m$ (resp. $\mu_n$) too because for any \'etale open $V \ra X$ we have 
\begin{eqnarray*} i^* \omega_{2,*} \GG_m(V)= \GG_m(\PP(\Omega_X \times_X V))=\Gamma(\PP(\Omega_X \times_X V))^{\times} = \Gamma(V)^{\times}= \GG_m(V)=i^* \GG_m(V) \\
i^* \omega_{2,*} \mu_n(V)= \mu_n(\PP(\Omega_X \times_X V))=\mu_n(\Gamma(\PP(\Omega_X \times_X V))) = \mu_n(\Gamma(V))= \mu_n(V)=i^* \mu_n(V)
\end{eqnarray*}
Since $\omega_2^{-1} (U) \ra U$ is an isomorphism, so the second vertical arrow is an isomorphism. Thus we see that the sheaves $\GG_m$ and $\omega_{2,*} \GG_m$ on $X^{(2)}$ have the same mapping cylinder of the left exact additive functor $i^* j_*:  \tilde{U}_{\acute{e}t} \ra \tilde{X}_{\acute{e}t}$ (cf \cite[page 136]{tamme}) where $\tilde{U}_{\acute{e}t}$ denotes the category of sheaves on the \'etale site of $U$. By the decomposition theorem \cite[Theorem 8.1.7]{tamme}, the \'etale sheaves $\omega_{2,*} \GG_m$ and $\GG_m$ are isomorphic. The same argument also shows that $\omega_{2,*} \mu_n$ is isomorphic to $\mu_n$.

%The argument for $\pi_2$ is very similar. We have $\pi_2^* \GG_m $ identifies with $\GG_m$. Using this identification, we have again a diagram like (\ref{decompositiondiag}):\begin{equation} \xymatrix{i^* \GG_m \ar[r] \ar[d] & i^* j_* j^* \GG_m \ar[d] \\i^* \pi_{2,*} \GG_m \ar[r] & i^* j_* j^* \pi_{2,*} \GG_m}\end{equation} We also have\begin{equation}i^* \pi_{2,*} \GG_m(V)= \GG_m(\pi_2^{-1}(V))= \GG_m(V),\end{equation}which shows $i^* \pi_{2,*} \GG_m$ identifies with $\GG_m$ on $X$.
\end{proof}

\subsection{Neron-Severi group of $X^{[2]}$ and $X^{(2)}$}
For a scheme $Z$ over a field $k$, let $\Pic(Z)$ denote the group of line bundles on $Z$. Let $\Pic^0(Z)$ denote the Picard scheme of $Z$ parametrizing algebraically trivial line bundles. Let \begin{equation}
NS(Z):=\Pic(Z)/\Pic^0(Z)(k)
\end{equation} denote the N\'eron-Severi group of $Z$.

We recall some results from \cite[\S 2, page 79-80]{knop} which hold over an algebraically closed field of arbitrary characteristic. Let an algebraic group $G$ act on an irreducible variety $X$. Let $L$ be a line bundle on $X$. A $G$-linearization on $L$ is a lift of $G$-action on $X$ to $L$ which is linear on the fibers. A $G$-line bundle is a line bundle on $X$ together with a $G$-linearization. Let $\Pic_G(X)$ denote the set of isomorphism classes of $G$-line bundles on $X$. By \cite[Lemma 2.2, page 80]{knop}  there is an exact sequence
\begin{equation} \label{forgetlin}
0 \ra H^1(G,\cO(X)^{\times}) \ra \Pic_G(X) \stackrel{forget}{\lra} \Pic(X)
\end{equation}

Now assume that $\pi: X \ra X//G$ admits a quotient. Then by \cite[\S 5, page 85]{knop}, we further have
\begin{equation} \label{lbquotlin}
1 \ra \Pic(X// G) \ra \Pic_G(X) \ra \prod_{x \in C} \XX^* (G_x)
\end{equation}
where $C$ is a set of representatives of closed orbits in $X$ and $G_x$ is the isotropy subgroup at $x$ and $\XX^*(G_x)$ is the group of characters. We also have the following exact sequence by \cite[\S 5.1]{knop}
\begin{equation} \label{5.1}
 E(X//G) \hra E(X)^G \ra \XX^*(G) \ra H^1(G, \cO(X)^{\times} ) \ra H^1(G/G^0,E(X)) 
\end{equation}
where $G^0$ is the connected component of the neutral element, $E(X)=\cO(X)^{\times}/k^{\times}$. 

Although these results are stated over an algebraically closed field of characteristic zero but we want to apply them to our situation of $\pi: X^{\times 2} \ra X^{(2)}$ where these hold for characteristic other  than two.

We need the following proposition for the case $d=2$ in this section but $d \geq 3$ in \S \ref{dgeq3}. So we state and prove it in general.
\begin{prop} \label{A1} Let $char(k) \neq 2$. We have  
\begin{eqnarray} 
\label{pichilbsym2} \label{1} \Pic(X^{(d)})=\Pic(X^{\times d})^{S_d} \\
 \label{2} \Pic^0(X^{(d)})=\Pic^0(X) \\
\label{3} NS(X^{[d]})=NS(X^{(d)}) \oplus \ZZ.
\end{eqnarray}
\end{prop}
\begin{proof} By \cite[Lemma 6.1]{fogarty}, we have a surjection $\Pic(X^{(d)}) \ra \Pic(X^{\times d})^{S_d}$. This means that every line bundle in $\Pic(X^{\times d})^{S_d}$ admits at least one  linearization. The action of $S_d$ on $\Gamma(X^{\times d})=k$ is trivial. Now in (\ref{5.1}) we have $E(X^{\times d})=\cO(X^{\times d})^{\times}/k^{\times}=k^\times/k^\times$ is trivial. Therefore the natural map of (\ref{5.1})
\begin{equation}
\XX^*(S_d) \ra H^1(S_d, \cO(X^{\times d})^{\times})=H^1(S_d,k^{\times})
\end{equation}
is an isomorphism. Using these facts (\ref{forgetlin}) simplifies to give 
\begin{equation} \label{picwithlin} 0 \ra \XX^*(S_d) \ra \Pic_{S_d}(X^{\times d}) \ra \Pic(X^{\times d})^{S_d} \ra 0
\end{equation} 

Since the branch locus for $\pi: X^{\times d} \ra X^{(d)}$ is connected, so  (\ref{lbquotlin}) becomes simply the following exact sequence \begin{equation} \label{piccomparison}
\xymatrix{
1 \ra \Pic(X^{(d)}) \ra \Pic_{S_d}(X^{\times d}) \stackrel{\alpha}{\ra} \XX^*(S_d) \ar@{.>}[r]  & 1.
}
\end{equation}

By (\ref{forgetlin}), $H^1(S_d,\cO(X^{\times d})^{\times})$, which equals $\XX^*(S_d)$, identifies with the set of linearizations on the trivial line bundle over $X^{\times d}$. On the other hand, the trivial line bundle on $X^{\times d}$ with  linearization given by the non-trivial character in $\XX^*(S_d)$ maps to the non-trivial character on $S_d$. Therefore the above sequence is also surjective.  Therefore (\ref{picwithlin}) is split exact  using $\alpha$. So it follows that \begin{equation} \label{nsymtwo}  \Pic(X^{(d)})= \Pic(X^{\times d})^{S_d} \quad \text{and hence} \quad NS(X^{(d)})=NS(X^{\times d})^{S_d}.
\end{equation}
This shows (\ref{1}).

Recall that $C_s(X)$ is the group of symmetric divisorial classes (cf \cite[page 678]{fogarty} for the case of general $d$).
Further by \cite[(6'), Lemma 6.1]{fogarty}, it follows that 
\begin{equation} \label{picinvtwo}
\Pic(X^{\times d})^{S_d}=\Pic(X) \times C_s(X).
\end{equation} 
Thus \begin{equation} \Pic^0(X^{(d)})=\Pic^0(X) 
\end{equation}  
This shows (\ref{2}).
Further
\begin{equation} \label{nsofsym}
NS(X^{(d)})=NS(X^{\times d})^{S_d}= NS(X) \times C_s(X).
\end{equation}

On the other  hand by \cite[Theorem 6.2]{fogarty}, for any $d$ we have 
\begin{equation} \label{picxd} \Pic(X^{[d]})= \Pic(X) \times C_s(X) \times \ZZ
\end{equation}   Thus by (\ref{nsymtwo}), (\ref{picinvtwo}) and (\ref{picxd}) we get
\begin{equation} 
\Pic(X^{[d]}) = \Pic(X^{(d)}) \times \ZZ.
\end{equation}

By \cite[Theorem 5.4]{fogarty}, for a smooth projective surface over any field $k$, for any $d$, we have an isomorphism  \begin{equation} \label{isopic0} \Pic^0(X/k) \ra \Pic^0(X^{[d]}/k). \end{equation}
Therefore $NS(X^{[d]})=NS(X) \times C_s(X) \times \ZZ$. Thus by (\ref{nsofsym}) we get
\begin{equation*}
NS(X^{[d]})=NS(X^{(d)}) \oplus \ZZ.
\end{equation*}
This shows (\ref{3}).

%Consider the morphism $\pi:X^{\times 2} \ra X^{(2)}$. Notice that $\pi_* \ul{\GG_m}=\ul{\GG_m}$ where $\ul{\GG_m}$ denotes the sheaf defined by $\GG_m$. Indeed, since for an open subset $U \hra X^{(2)}$ we have $\pi^{-1}(U)$ is connected, so $\pi_* \GG_m(U)=\GG_m(\pi^{-1}(U))=\cO(\pi^{-1}(U))^{\times}$ equals $\cO(U)^{\times}$ because $\pi$ is a finite morphism. Thus $\pi_* \ul{\GG_m}=\ul{\GG_m}$. Therefore by the Leray spectral sequence with values in $\GG_m$: \begin{equation}  H^p(X^{(2)},R^q \pi_* \ul{\GG_m}) \implies H^*(X^{\times 2}, \ul{\GG_m}) \end{equation} it follows that $\Pic(X^{(2)})$ injects into $\Pic(X^{\times 2})$. But the pull-back map factors through $\Pic(X^{\times 2})^{S_2}$.  Hence   $\Pic(X^{\times 2} )^{S_2}$ surjects onto $NS(X^{(2)})$.  By (\ref{picxd}), (\ref{picsurj}), the isomorphism $\Pic(X^{(2)}) \simeq \Pic(X^{\times 2})^{S_2}$ and (\ref{isopic0}), it follows that the $\ZZ$ rank of the N\'eron-Severi of $X^{[2]}$ is one more than that of $X^{(2)}$.

\end{proof}
\begin{rem} By the proof of \cite[Theorem 6.2]{fogarty}, we see that for $d=2$ the copy of $\ZZ$ in (\ref{picxd}) comes from the divisor corresponding to $\PP(\Omega_X)$.
\end{rem}
\subsection{Comparison of Cohomological Brauer group in any characteristic}

Let $k$ be an algebraically closed field of characteristic $p$. Let $\ell \neq p$ be a prime. Let $A$ be an abelian group. For any $m \geq 1$, let $A_{\ell^m}$ denote the subgroup of $\ell^m$-torsion points of $A$. Let $A_{\ell^*}$ denote the set of all elements of $A$ of order a power of $\ell$.

\begin{thm} \label{anychar} Let $char(k) \neq 2$. The morphism $\omega_2:X^{[2]} \ra X^{(2)}$ induces an isomorphism of cohomological Brauer groups $Br'(X^{(2)})_{\ell^*} \ra Br(X^{[2]})_{\ell^*}$.
\end{thm}
\begin{proof} It suffices to check the isomorphism for torsion elements of order $\ell^m$ for any $m \geq 1$. We set $n=\ell^m$ in the following. Consider the Leray-Serre spectral sequence for $\omega_2: X^{[2]} \ra X^{(2)}$ with values in the constant sheaf defined by $\mu_n$:
\begin{equation}
E^{p,q}_2=H^p(X^{(2)},R^q \omega_{2*} \ul{\mu_n}) \implies E^{p+q}=H^{p+q}(X^{[2]}, \ul{\mu_n}).
\end{equation}
Consider the $7$-term long sequence (cf \cite[page 371, Cor 3.2]{merkur})
\begin{equation} \label{seventerm} 0 \ra E^{1,0}_2 \ra E^1 \ra E^{0,1}_2 \ra E^{2,0}_2 \ra ker(E^2 \ra E^{0,2}_2) \ra E^{1,1}_2 
\end{equation}
Now for $p>0$ the sheaf $R^p \omega_{2*} \mu_n$ vanishes on the complement $j:U  \ra X^{(2)}$ of the closed subscheme $i: X \ra X^{(2)}$ by the proper base  change theorem. So by the decomposition theorem \cite[Theorem 8.1.2]{tamme} the natural map 
\begin{equation}
R^p \omega_{2*} \mu_n \ra i_* i^* R^p \omega_{2*} \mu_n
\end{equation}
of \'etale sheaves is an isomorphism for any $p>0$. Further for $p>0$ by the proper base change theorem we have an isomorphism of \'etale sheaves
\begin{equation}
i^* R^p \omega_{2*} \mu_n \ra R^p \pi_* j^* \mu_n.
\end{equation}
Now we also have an isomorphism of \'etale sheaves $j^* \ul{\mu_n}=\ul{\mu_n}$. Therefore we have an isomorphism of groups for $i \in \{0,1\}$:
\begin{equation} \label{lowvanish}
H^i(X^{(2)},R^1 \omega_{2*} \mu_n) =H^i(X,R^1 \pi_* \mu_n)
\end{equation}
 The stalks of the sheaves $R^1 \pi_* \mu_n$ can be computed by considering the fiber of $\pi$. By \cite[Prop 10.3.3]{tamme} we have the exact sequence
\begin{equation*}
0 \ra H^0(\PP^1, \mu_n) \ra H^0(\PP^1, \GG_m) \ra H^0(\PP^1,\GG_m) \ra H^1(\PP^1, \mu_n) \ra \Pic(\PP^1) \stackrel{n}{\ra} \Pic(\PP^1)
\end{equation*} 
 Since each fiber of $\pi$ identifies with $\PP^1$ which has no torsion line bundles and $k$ is algebraically closed, so $H^1(\PP^1,\mu_n)=0$. Thus 
 \begin{equation} \label{r1omegavanish}
 R^1 \omega_{2*} \mu_n =R^1 \pi_* \mu_n =0,
 \end{equation}
 and the groups in (\ref{lowvanish}) vanish. This means that $E^{0,1}_2$ and $E^{1,1}_2$ in (\ref{seventerm}) vanish. 

Thus we get the exact sequence 
\begin{equation} \label{munseq}
0 \ra H^2(X^{(2)}, \omega_{2,*} \mu_n) \ra H^2(X^{[2]},\mu_n) \ra H^0(X^{(2)},R^2 \omega_{2*} \mu_n)
\end{equation} 
We will now simplify this sequence and show that it is also surjective after two preparations.
The sheaf  $\omega_2^* \ul{\mu_n}$ identifies with the sheaf defined $\mu_n$ on $X^{[2]}$ by Proposition \ref{pushforwardbyomega2}.  Reasoning for $R^2 \pi_* \mu_n$ as we have done for $R^1 \pi_* \mu_n$ we get
\begin{equation}
H^0(X^{(2)},R^2 \omega_{2*} \mu_n )=H^0(X,R^2 \pi_* \mu_n).
\end{equation}
Since the fibers of $\pi$ are isomorphic to $\PP^1$, so by \cite[Prop 10.3.3]{tamme} we note over $x \in X$ the stalk $(R^2 \pi_* \mu_n)_x$ is given by $
\Pic(\pi^{-1}(x)) \stackrel{n}{\ra} \Pic(\pi^{-1}(x)) \ra H^2(\pi^{-1}(x),\mu_n) \ra 0$. Hence it is discrete being isomorphic to $ \ZZ/n\ZZ$.
Further this isomorphism is canonical because the ample generator of $\PP(\Omega_X)_x = \PP(\Omega_{X,x})= \PP^1$ maps to $1 \in \ZZ/n\ZZ$. Finally since $\pi$ is a locally trivial $\PP^1$-fibration, so over the overlaps of \'etale open sets, the ample generator of $\PP^1$ must map to itself. Thus $R^2 \pi_* \mu_n$ identifies with the constant sheaf defined by $\ZZ/n\ZZ$. Hence we have
\begin{equation} \label{evaluationr2mun}
H^0(X^{(2)},R^2 \omega_{2*} \mu_n )=H^0(X,R^2 \pi_* \mu_n)=\ZZ/n\ZZ.
\end{equation}

 %Now $d^{0,2}_2: E^{0,2}_2 \ra E^{2,1}_2$ is zero because $R^1 \pi_* \mu_n=0$. So $E^{0,2}_2=E^{0,2}_3$. Now consider $d^{0,2}_3: E^{0,2}_3 \ra E^{3,0}_3$. Now by $E^{1,1}_2 \ra E^{3,0}_2 \ra E^{5,-1}_2$, we see that $E^{3,0}_3=E^{3,0}_2=H^3(X^{(2)}, \omega_{2,*} \mu_n)$. Therefore using the differential $d^{0,2}_3$ we have \begin{equation*}E^{0,2}_\infty=\ker(H^0(X,R^2\pi_* \mu_n)=E^{0,2}_2=E^{0,2}_3 \ra E^{3,0}_3=E^{3,0}_2=H^3(X^{(2)},\mu_n)).\end{equation*} Let us denote $E^{0,2}_\infty$ as $E^{0,2}_\infty(n)$. Taking colimits we get the sequence\begin{equation} \varinjlim_n H^2(X^{(2)},\mu_n) \ra \varinjlim_n H^2(X^{[2]},\mu_n) \ra \varinjlim_n E^{0,2}_\infty(n) \ra 0\end{equation}Here exactness on the right follows from surjectivity of $H^2( X^{[2]},\mu_n) \ra E^{0,2}_\infty(n)$.By (\ref{munseq}), it is also exact on the left.  

We now come to the second preparation.
Consider the Kummer sequence 
\begin{equation} \label{kummerseq} 1 \ra \mu_n \ra \GG_m \stackrel{n}{\ra} \GG_m \ra 1
\end{equation} of \'etale sheaves on $X^{[2]}$.  So taking pushforward by $\omega_{2}$, we get the exact sequence $0 \ra \omega_{2*} \mu_n \ra \omega_{2*} \GG_m \ra \omega_{2*} \GG_m \ra R^1 \omega_{2*} \mu_n \ra R^1 \omega_{2*} \GG_m \ra R^1 \omega_{2*} \GG_m \ra R^2 \omega_{2*} \mu_n \ra \cdots$. Recall that $\GG_m \simeq \omega_{2*} \GG_m$ and $\mu_n \simeq \omega_{2*} \mu_n$ by Proposition \ref{pushforwardbyomega2}. So the begining of the long exact sequence identifies with the Kummer sequence of sheaves on the \'etale site of $X^{(2)}$. So it is also right exact. Thus  by (\ref{r1omegavanish}), we get an exact sequence of \'etale sheaves
\begin{equation}
0 \ra R^1 \omega_{2*} \GG_m \stackrel{R^1(n)}{\lra} R^1 \omega_{2*} \GG_m \ra R^2 \omega_{2*} \mu_n
\end{equation} 
The map $R^1(n)$ is also multiplication by $n$ on the sheaf $R^1 \omega_{2*} \GG_m$. Indeed, if we take an injective resolution of $\GG_m \ra I^{\bullet}$, then $n: \GG_m \ra \GG_m$ lifts to a map of complexes $I^{\bullet} \ra I^{\bullet}$ which is multiplication by $n$ on each piece and thus induces multiplication by $n$ on cohomology sheaves. By (\ref{evaluationr2mun}), we get
\begin{equation} \label{gammar1n}
0 \ra H^0(X^{(2)},R^1 \omega_{2*} \GG_m) \stackrel{\Gamma(R^1(n))}{\lra} H^0(X^{(2)},R^1 \omega_{2*} \GG_m) \ra \ZZ/n\ZZ
\end{equation}
and $\Gamma(R^1(n))$ is multiplication by $n$. By the Leray sequence for $\omega_2:X^{[2]} \ra X^{(2)}$ with values in $\GG_m$, in low degree terms we have the exact sequence
$0 \ra H^1(X^{(2)},\omega_{2*} \GG_m) \ra H^1(X^{[2]},\GG_m) \ra
 H^0(X^{(2)},R^1 \omega_{2*} \GG_m) \ra H^2(X^{(2)},\omega_{2*} \GG_m) \stackrel{\delta}{\ra} H^2(X^{[2]}, \omega_{2*} \GG_m)$

Recall $\omega_{2*} \GG_m \simeq \GG_m$ by Proposition \ref{pushforwardbyomega2} and 
  that by (\ref{pichilbsym2}) we have 
  \begin{equation} \label{alphaissplitsurjection} \Pic(X^{[2]}) = \Pic(X^{(2)}) \oplus \ZZ
  \end{equation} where $\ZZ$ corresponds to the line bundle of the  divisor $\PP(\Omega_X) \hra X^{[2]}$.
So the cokernel of $0 \ra \Pic(X^{(2)}) \ra \Pic(X^{[2]})$ identifies with $\ZZ$. Now we view $\ZZ$ as a subgroup of $H^0(X^{(2)},R^1 \omega_{2*} \GG_m)$. Thus  we get the exact sequence
\begin{equation} \label{intermediatesequence}
0 \ra \ZZ \ra H^0(X^{(2)},R^1 \omega_{2*} \GG_m) \ra H^2(X^{(2)},\GG_m) \stackrel{\delta}{\ra} H^2(X^{[2]}, \GG_m)
\end{equation}

Let us analyse the consequences of the inclusion $\ZZ \hra H^0(X^{(2)},R^1 \omega_{2*} \GG_m)$ on (\ref{gammar1n}) where $\Gamma(R^1(n))$ acts by multiplication by $n$. It follows that (\ref{gammar1n}) must be surjective, the induced map on $H^0(X^{(2)},R^1 \omega_{2*} \GG_m)/\ZZ$ is multiplication by $n$ and it is an isomorphism. Thus in (\ref{intermediatesequence})  we get that $\ker(\delta)$ has no non-trivial element of order dividing $n$.

Let us now prove that (\ref{munseq}) is surjective.   Now consider the diagram with exact rows and columns obtained from the Kummer sequence where $\Pic()/n$ denotes $\Pic()/n\Pic()$ where the middle vertical column is (\ref{munseq}):
\begin{equation*}
\xymatrix{
\Pic(X^{(2)})/n \ar@{^{(}->}[r] \ar@{^{(}->}[d] & H^2(X^{(2)},\mu_n) \ar[r] \ar@{^{(}->}[d] & H^2(X^{(2)},\GG_m) \ar[d]^{\delta} \\
\Pic(X^{[2]})/n \ar@{^{(}->}[r] \ar@{->>}[d]^{\alpha} & H^2(X^{[2]},\mu_n) \ar[r] \ar[d]^{\gamma} & H^2(X^{[2]},\GG_m)  \\
  \ZZ/n \ar[r]^{\beta}  & H^0(X^{(2)},R^2 \omega_{2,*} \mu_n) &  \\
}
\end{equation*}
We want to prove that $\beta$ is an isomorphism. Then a little diagram chase would show that (\ref{munseq}) is surjective. By (\ref{evaluationr2mun}), it suffices to show that $\beta$ is injective. Let $e \in \ZZ/n\ZZ$ be a non-zero element in the kernel of $\beta$.

Since  $\alpha$ is a split surjection by (\ref{alphaissplitsurjection}), so $e$ lifts to $H^2(X^{[2]},\mu_n)$. Since $\gamma(e)=0$, so it defines a non-zero element, say $e' \in H^2(X^{(2)},\GG_m)$ lying in the kernel of $\delta$. We showed earlier that $\ker(\delta)$ does not have  any element of order $n$. But $e'$ such an element. So it must be zero. Hence we get a contradiction. Thus $\beta$ is injective and so an isomorphism. Therefore $\gamma$ is surjective.
Thus from (\ref{munseq}) we get the following sequence is exact:
\begin{eqnarray} \label{uptomun}
\quad  \quad 0 \ra  H^2(X^{(2)},\mu_{\ell^m}) \ra  H^2(X^{[2]},\mu_{\ell^m}) \ra  H^0(X^{(2)},R^2 \omega_{2*} \mu_{\ell^m}) \ra 0 
\end{eqnarray}

%Now $ H^0(X^{(2)},R^2 \omega_{2,*} \mu_{l^m})= \ZZ/l^m\ZZ$. Taking $\varinjlim_m$ we get\begin{equation}0 \ra \varinjlim_m H^2(X^{(2)},\mu_{l^m}) \ra \varinjlim_m H^2(X^{[2]},\mu_{l^m}) \ra \ZZ_l \ra 0\end{equation}

By the Kummer sequence, for any scheme $Z$ we get the exact sequence
 \begin{eqnarray} 
\label{formun}  0 \ra NS(Z) \otimes \ZZ/\ell^m\ZZ \ra H^2_{\acute{e}t}(Z,  \mu_{\ell^m}) \ra Br'(Z)_{\ell^m} \ra 0 
  \end{eqnarray}
  By  Proposition \ref{A1} we have 
  \begin{equation} \label{nsmodn} NS(X^{[2]})/n=NS(X^{(2)})/n \oplus \ZZ/n.
  \end{equation} 
    
  So since $n=\ell^m$, thus by (\ref{uptomun}), taking in account (\ref{formun}) for $X^{[2]}$ and $X^{(2)}$, by (\ref{nsmodn}) and (\ref{evaluationr2mun}) it follows by
\begin{equation}
\xymatrix{
0 \ar[r] & Br'(X^{(2)})_{\ell^m} \ar@{.>}[r] & Br'(X^{[2]})_{\ell^m} \ar[r] & 0 \ar[r]  & 0 \\
0 \ar[r] & H^2(X^{(2)},\mu_{\ell^m}) \ar[r] \ar@{>>}[u] & H^2(X^{[2]},\mu_{\ell^m} \ar[r] \ar@{>>}[u] & H^0(X^{(2)},R^2 \omega_{2*} \GG_m) \ar[r] \ar[u] & 0 \\
0 \ar[r] & NS(X^{(2)}) \otimes \ZZ/\ell^m \ar[r] \ar@{^{(}->}[u] & NS(X^{[2]}) \otimes \ZZ/\ell^m \ar[r] \ar@{^{(}->}[u] & \ZZ/\ell^m \ar@{.>}[u] \ar[r] & 0
}
\end{equation}  
 that the dotted vertical arrow is an isomorphism. Thus the natural morphism $Br'(X^{(2)})_n \ra Br'(X^{[2]})_n$ is an isomorphism. 
 \end{proof}

%Here the fibers of $\pi$ identify with $\PP^1_\CC$. Viewing this diagram in the category of topological spaces, the inclusion $j$ is a cofibration. Thus the diagram is a homotopy push-out square. It gives rise to the Mayer-Vietoris sequence. The terms relevant to us are\begin{equation*}  H_3(X^{(2)}) \stackrel{\alpha}{\ra} H_2(\PP)  \ra H_2(X) \oplus H_2(X^{[2]}) \ra  H_2(X^{(2)}) \stackrel{\beta}{\ra} H_1(\PP)  \ra H_1(X) \oplus H_1(X^{[2]})\end{equation*} Here we have  omitted the coefficent ring is $\ZZ$. The map $\pi$ induces an isomorphism on fundamental group. In particular it induces an isomorphism on $H_1$. So $\beta=0$. Thus we get the  exact sequence\begin{equation} \label{relevantterms}H_3(X^{(2)}) \stackrel{\alpha}{\ra} H_2(\PP) \ra H_2(X) \oplus H_2(X^{[2]}) \ra H_2(X^{(2)}) \ra 0\end{equation}

\section{Cohomological Brauer group of $X^{(2)}$ in terms of $X^{\times 2}$} \label{cohbrgpsym2}  We fix an algebraically closed field $k$ of characteristic  $p$.  Let $\ell \neq p $ be a prime. For an abelian group $A$ let $A_{\ell^m}$ denote the subgroup of ${\ell}^m$ torsion elements of $A$. Let $A_{\ell^*}$ denote the set of $\ell^m$ torsion elements for all $m \geq 1$.

Let $S_2$ act on $X^{\times 2}:=X \times X$ by switching the two factors. Let 
\begin{equation}
[X^{\times 2}/S_2]
\end{equation}
denote the stack quotient. Consider the cartesian diagram
\begin{equation}
\xymatrix{
\cG \ar[r]_{j} \ar[d]^{\pi'} & [X^{\times 2}/S_2] \ar[d]^{\pi} \\
X \ar[r]^{i} & X^{(2)}
}
\end{equation}
where $\cG$ denotes the fibered product. The gerbe $\pi: \cG \ra X $ is neutral (\cite[D\'efinition 3.20]{lmb}) because $\Delta: X \ra X^{\times 2} \ra X^{(2)}$ provides a section $s: X \ra \cG$. Now the group scheme $G\ra X^{(2)}$ whose sections on $y: U \ra X$ is given by $Isom(s(y),s(y))$ identifies with $S_2 \times X$ because the identity morphism $ X \ra X$ factors through $X \ra \cG \ra X$. Thus if $BS_2:=[\Spec(k)/S_2]$, then $B(G/X) \simeq BS_2 \times X$ as classifying spaces and therefore by \cite[Lemme 3.21]{lmb} we have the following isomorphism of $S_2$-gerbes on $X$
\begin{equation}
\cG \simeq X \times BS_2
\end{equation}
Thus the sheaf $R^k \pi'_* \mu_n$ identifies with the constant sheaf with fiber $H^k(BS_2, \mu_n)$. We will denote this isomorphism of sheaves as follows:
\begin{equation} \label{isoconstantsheaf}
R^k \pi'_* \mu_n = \ul{H^k(BS_2,\mu_n)}.
\end{equation}

\begin{prop} \label{pushforwardbypi} The natural adjunction maps $\GG_m \ra \pi_* \pi^*
 \GG_m$ and $\mu_n \ra \pi_* \pi^* \mu_n$ are isomorphisms of \'etale sheaves. 
%Consider the natural morphism $\pi_2: X^{\times 2} \ra X^{(2)}$. The natural adjunction map $\GG_m \ra \pi_{2,*} \pi_2^* \GG_m$ is an isomorphism of \'etale sheaves on $X^{(2)}$.
\end{prop}
\begin{proof} Now $\pi^* \GG_m$ (resp $\pi^* \mu_n$) identifies with the sheaf defined by $\GG_m$ (resp. $\mu_n$) on $[X^{\times 2}/S_2]$. Here below after using this indentification, we get the following commutative diagram where all vertical arrows are obtained from adjunction $Id \ra \pi_* \pi^*$ and the horizontal by $Id \ra j_* j^*$:
\begin{equation} \label{decompositiondiag2}
\xymatrix{
i^* \GG_m \ar[r] \ar[d] & i^* j_* j^* \GG_m \ar[d] \\
i^* \pi_* \GG_m \ar[r] & i^* j_* j^* \pi_* \GG_m
} \quad
\xymatrix{
i^* \mu_n \ar[r] \ar[d] & i^* j_* j^* \mu_n \ar[d] \\
i^* \pi_* \mu_n \ar[r] & i^* j_* j^* \pi_* \mu_n 
}
\end{equation}

Now the left vertical arrows are an isomorphism because  $i^* \pi_* \GG_m$ (resp. $i^* \pi_* \mu_n$) identifies with $\GG_m$ (resp. $\mu_n$) too because for any \'etale open $V \ra X$ we have 
\begin{eqnarray*} i^* \pi_* \GG_m(V)= \GG_m(BS_2 \times V)=\Gamma(BS_2 \times V)^{\times} = \Gamma(V)^{\times}= \GG_m(V)=i^* \GG_m(V) \\
i^* \pi_* \mu_n(V)= \mu_n(BS_2 \times V)=\mu_n(\Gamma(BS_2 \times V)) = \mu_n(\Gamma(V))= \mu_n(V)=i^* \mu_n(V)
\end{eqnarray*}
Since $\pi^{-1} (U) \ra U$ is an isomorphism, so the second vertical arrow is an isomorphism. Thus we see that the sheaves $\GG_m$ and $\pi_* \GG_m$ on $X^{(2)}$ have the same mapping cylinder of the left exact additive functor $i^* j_*:  \tilde{U}_{\acute{e}t} \ra \tilde{X}_{\acute{e}t}$ (cf \cite[page 136]{tamme}) where $\tilde{U}_{\acute{e}t}$ denotes the category of sheaves on the \'etale site of $U$. By the decomposition theorem \cite[Theorem 8.1.7]{tamme}, the \'etale sheaves $\pi_* \GG_m$ and $\GG_m$ are isomorphic. The same argument also shows that $\pi_* \mu_n$ is isomorphic to $\mu_n$.

%The argument for $\pi_2$ is very similar. We have $\pi_2^* \GG_m $ identifies with $\GG_m$. Using this identification, we have again a diagram like (\ref{decompositiondiag}):\begin{equation} \xymatrix{i^* \GG_m \ar[r] \ar[d] & i^* j_* j^* \GG_m \ar[d] \\i^* \pi_{2,*} \GG_m \ar[r] & i^* j_* j^* \pi_{2,*} \GG_m}\end{equation} We also have\begin{equation}i^* \pi_{2,*} \GG_m(V)= \GG_m(\pi_2^{-1}(V))= \GG_m(V),\end{equation}which shows $i^* \pi_{2,*} \GG_m$ identifies with $\GG_m$ on $X$.
\end{proof}

We will switch to the Big-\'etale site while working over stacks because pull-back does not commute with finite limits over the small and lisse-\'etale sites \cite{olsson} but it does for the Big-\'etale site. This should not cause a problem because for schemes the cohomology  over the big and small \'etale sites agree \cite{tamme}.
  
  We begin by proving the analogue for the quotient stack $[X^{\times 2}/S_2]$ of a well known sequence for schemes $Z$ such that $\Pic^0(Z)$ is divisible.
  
\begin{prop} For any $n \geq 1$ we have the following exact sequence
\begin{equation} \label{kummermunbrn}
0 \ra NS([X^{\times 2}/S_2]) \otimes \ZZ/n \ra H^2_{\acute{E}t}([X^{\times 2}/S_2],\mu_n) \ra Br'([X^{\times 2}/S_2])_n \ra 0
\end{equation}
\end{prop} 
\begin{proof}In the proof of Proposition \ref{A1} we showed that the sequence (\ref{piccomparison}) $$1 \ra \Pic(X^{(2)}) \ra \Pic_{S_2}(X^{\times 2}) \ra \XX^*(S_2) \ra 1$$ is exact and split by mapping the non-trivial character $\XX^*(S_2)$ to the non-trivial linearization on the trivial bundle. Now a line bundle on the quotient stack $[X^{\times 2}/S_2]$ is equivalently a line bundle on $X^{\times 2}$ together with a $S_2$-linearization. Thus 
\begin{equation} \label{stackwithlin} \Pic_{S_2}(X^{\times 2})= \Pic([X^{\times 2}/S_2]).
\end{equation} So (\ref{piccomparison}) becomes
\begin{equation} \label{splitstackpic}
0 \ra \Pic(X^{(2)}) \ra \Pic([X^{\times 2}/S_2]) \ra \XX^*(S_2) \ra 1.
\end{equation}
This shows that 
\begin{equation}
\Pic^0([X^{\times 2}/S_2])=\Pic^0(X^{(2)}),
\end{equation}
in particular they are divisible. By definition of the N\'eron-Severi group we have the exact sequence $0 \ra \Pic^0([X^{\times 2}/S_2]) \ra \Pic([X^{\times 2}/S_2]) \ra NS([X^{\times 2}/S_2]) \ra 0$. Thus tensoring with $\ZZ/n$ we get
\begin{equation} \label{picnstensorn}
\Pic([X^{\times 2}/S_2]) \otimes \ZZ/n = NS([X^{\times 2}/S_2]) \otimes \ZZ/n.
\end{equation} By taking the long exact sequence associated to the Kummer sequence (\ref{kummerseq}) and analyzing the cohomology terms we observe that $H^1([X^{\times 2}/S_2],\GG_m)=\Pic([X^{\times 2}/S_2])$ and the induced map is multiplication by $n$ on $H^i([X^{\times 2}/S_2],\GG_m)$. This gives $$
\Pic([X^{\times 2}/S_2]) \stackrel{n}{\ra} \Pic([X^{\times 2}/S_2]) \ra H^2([X^{\times 2}/S_2],\mu_n) \ra H^2([X^{\times 2}/S_2],\GG_m) \stackrel{n}{\ra}$$ which implies $$0 \ra coker(n) \ra H^2([X^{\times 2}/S_2],\mu_n) \ra \ker(n) \ra 0.$$  The cokernel on $\Pic$ by (\ref{picnstensorn}), identifies with $NS([X^{\times 2}/S_2]) \otimes \ZZ/n$ and kernel of the second $n$ identifies with $Br'([X^{\times 2}/S_2])_n$. Thus we get (\ref{kummermunbrn}).
\end{proof}

\begin{thm} \label{stacktosym2} For $\ell \neq 2$ a prime, 
the following natural map is an isomorphism 
\begin{equation} \label{oddsym2stack} Br'(X^{(2)})_{\ell^*} \ra Br'([X^{\times 2}/S_2])_{\ell^*}.
\end{equation}
 For $\ell=2$ we  have the following exact sequence
\begin{equation}  \label{2*}
0 \ra Br'(X^{(2)})_{2^*} \ra  \ker(Br'([X^{\times 2}/S_2])_{2^*} \ra \mu_2) \ra H^1(X,\mu_2) 
\end{equation}
\end{thm}
\begin{proof}   We begin with a preparation which holds both for $\ell \neq 2$ and $\ell=2$.

%Consider the Leray spectral sequence with values in $\GG_m$ for $\pi: X^{\times 2} \ra [X^{\times 2}/S_2]$. Since $\pi_* \GG_m=\GG_m$ by Proposition \ref{pushforwardbypi}, so by (\ref{splitstackpic}), the $7$-term sequence associated to the Leray sequence for $\pi$ with values in $\GG_m$ solves to \begin{equation}  \label{7to5term} 0  \ra \XX^*(S_2) \ra E^{0,1}_2 \ra E^{2,0}_2 \ra ker(E^2 \ra E^{0,2}_2) \ra E^{1,1}_2 \end{equation}
Let $n=l^m$ below. Consider the Leray spectral sequence with values in $\mu_n$ for $\pi$: 
$$E^{p,q}_2=H^p(X^{(2)},R^q \pi_* \mu_n) \implies H^{p+q}([X^{\times 2}/S_2], \mu_n).$$ Taking the associated $7$-term sequence, we get
\begin{equation} \label{7termformun}
0 \ra E^{1,0}_2 \ra E^1 \ra E^{0,1}_2 \ra  E^{2,0}_2 \ra ker(E^2 \ra E^{0,2}_2) \ra E^{1,1}_2
\end{equation}

Since $\pi^{-1}(U) \ra U$ is an isomorphism, so by proper base change theorem it follows that for $p >0$ the sheaves $R^p \pi_* \mu_n$ are supported on the closed subscheme $i: X \hra X^{(2)}$. So by decomposition theorem the natural morphism of \'etale sheaves $$R^p \pi_* \mu_n \ra i_* i^* R^p \pi_* \mu_n$$ is an isomorphism. Further by the proper base-change theorem in \'etale cohomology we have the isomorphism of \'etale sheaves
\begin{equation}
i^* R^p \pi_* \mu_n \ra R^p \pi'_* j^* \mu_n. 
\end{equation}
Since $j^* \mu_n=\mu_n$, so we get the isomorphism 
\begin{equation} R^p \pi_* \mu_n \ra i_* R^p \pi'_* \mu_n.
\end{equation}
Now the sheaf of groups $R^1\pi'_*\mu_n$ has fibers isomorphic to the group $H^1(BS_2,\mu_n)$. Since the cohomology of the classifying space equals that of the group, so this equals $H^1(S_2,\mu_n)$ which equals the two torsion elements in $\mu_n$.  So 
\begin{eqnarray} \label{r1piprime}
R^1 \pi'_* \mu_n= \mu_2 \quad \text{for} \quad \ell=2 \\
 0 \quad \text{for} \quad \ell \neq 2.
\end{eqnarray}
 Now  $ E^{0,1}_2= H^0_{\acute{E}t}(X^{(2)},R^1 \pi_* \mu_n) =H^0_{\acute{E}t}(X,R^1 \pi'_* \mu_n) \stackrel{(\ref{isoconstantsheaf})}{=}H^0(X, \ul{H^1(BS_2,\mu_n)})= H^1 (BS_2, \mu_n)= H^1(S_2,\mu_n)$, which equals $\mu_2$ for $l=2$ and vanishes otherwise. Therefore 
\begin{eqnarray} \label{e012}
E^{0,1}_2 & =& \mu_2 \quad \ell=2 
\\ &=& 0 \quad \quad  \ell \neq 2
\end{eqnarray} Recall that $n=\ell^m$. Further by (\ref{r1piprime}) we get
\begin{eqnarray} \qquad \quad E^{1,1}_2 = H^1_{\acute{E}t}(X^{(2)},R^1 \pi_* \mu_n)=H^1_{\acute{E}t}(X,R^1 \pi'_* \mu_n) &=& H^1_{\acute{E}t} (X, \mu_2) \quad  \ell=2 \\
&=& 0 \qquad \quad \ell \neq 2.
\end{eqnarray}
Similarly, $E^{0,2}_2=H^0_{\acute{E}t}(X^{(2)},R^2 \pi_* \mu_n)=H^0_{\acute{E}t}(X,R^2 \pi'_* \mu_n)=H^2_{\acute{E}t} (BS_2, \mu_n)$. So this equals $H^2(S_2,\mu_n)=\mu_n/2\mu_n$. This last term equals the $2$-torsion points of $\mu_n$. So
\begin{eqnarray} \label{e012twocases}
E^{0,2}_2= \mu_2 \quad  \text{for} \quad \ell=2 \\  0 \quad \text{for} \quad \ell \neq 2.
\end{eqnarray}

 Putting these terms in  the sequence (\ref{7termformun}), for $\ell \neq 2$ we get the isomorphism
\begin{equation} \label{oddiso}
H^2(X^{(2)},\mu_n) \stackrel{\simeq}{\ra} H^2([X^{\times 2}/S_2],\mu_n) 
\end{equation} 
 For $\ell=2$, by taking $n$-torsion points the exact sequence (\ref{splitstackpic})  gives 
 \begin{equation}
0 \ra \Pic(X^{(2)})_n \ra \Pic([X^{\times 2}/S_2])_n \ra \XX^*(S_2) \ra 1.
\end{equation}
Since $\pi_* \mu_n=\mu_n$ by Proposition \ref{pushforwardbypi}, so 
in the sequence (\ref{7termformun}) we have $E^{1,0}_2=\Pic(X^{(2)})_n$ and $E^1=\Pic([X^{\times 2}/S_2])_n$. Since $E^{0,1}_2=\mu_2$ for $\ell=2$ by (\ref{e012}),  so for $\ell=2$ the first three terms of (\ref{7termformun}) form a short exact sequence. So (\ref{7termformun}) reduces to 
\begin{equation} \label{longseqwithchar}
0 \ra H^2_{\acute{E}t}(X^{(2)},\mu_n) \ra \ker(H^2_{\acute{E}t}([X^{\times 2}/S_2],\mu_n) \ra \mu_2) \ra H^1(X, \mu_2)
\end{equation}

Now we want to pass from $\mu_n$ to $\GG_m$ coefficients. Since (\ref{splitstackpic}) is split exact it follows that 
\begin{equation} 0 \ra NS(X^{(2)})   \ra NS([X^{\times 2}/S_2]) \ra \XX^*(S_2) \ra 0
\end{equation} is split exact. Tensoring with $\ZZ/n$ that 
\begin{eqnarray} \label{oddns} NS(X^{(2)}) \otimes \ZZ/n\ZZ \ra NS([X^{\times 2}/S_2]) \otimes \ZZ/n\ZZ \quad \ell \neq 2 
\end{eqnarray}
is an isomorphism and 
\begin{eqnarray}
\label{bottom} \quad \quad  \quad 0 \ra  NS(X^{(2)}) \otimes \ZZ/n\ZZ \ra NS([X^{\times 2}/S_2]) \otimes \ZZ/n\ZZ \ra \mu_2 \ra 0 \quad \ell=2
\end{eqnarray}
is   a split exact sequence.
 On the other hand, by the Kummer exact sequence for $X^{(2)}$ and by ( \ref{kummermunbrn}) for $[X^{\times 2}/S_2]$, we  have exact sequences
\begin{eqnarray*}
0 \ra NS(X^{(2)}) \otimes \ZZ/n\ZZ \ra H^2(X^{(2)},  \mu_n ) \ra Br'(X^{(2)})_n \ra 1 \\  0 \ra NS([X^{\times 2}/S_2]) \otimes \ZZ/n\ZZ \ra H^2([X^{\times 2}/S_2],\mu_n ) \ra Br'([X^{\times 2}/S_2])_n \ra 1
\end{eqnarray*}
Recall that $n=\ell^m$. Thus quotienting  (\ref{oddiso}) by (\ref{oddns}) for $\ell \neq 2$ we get the isomorphism
\begin{equation}
H^2(X^{(2)},\GG_m)_n \ra H^2([X^{\times 2}/S_2],\GG_m)_n.
\end{equation}
Since this holds for any $n=\ell^m$, so we get (\ref{oddsym2stack}).
This settles the  case $\ell \neq 2$. For $\ell=2$ we argue as follows. By (\ref{longseqwithchar}) and (\ref{bottom}) consider the image of $\mu_2$ under the composite
\begin{equation} \label{casebycase}
0 \ra \mu_2 \stackrel{(\ref{bottom})}{\lra} \frac{H^2_{\acute{E}t}([X^{\times 2}/S_2],\mu_n)}{NS(X^{(2)}) \otimes \ZZ/n\ZZ} \stackrel{(\ref{longseqwithchar})}{\lra} \mu_2
\end{equation}
Now either the composite is zero, or it is non-zero. If it is zero, then it factors through $H^2_{\acute{E}t}([X^{\times 2}/S_2],\GG_m)_{n} \ra \mu_2$. In the second case,  the $\mu_2$ inclusion is split by the second map in (\ref{casebycase}) given by (\ref{longseqwithchar}).  Thus (\ref{longseqwithchar}) factors as $$\frac{H^2([X^{\times 2}/S_2],\mu_n)}{NS(X^{(2)}) \otimes \ZZ/n\ZZ}=H^2_{\acute{E}t}([X^{\times 2}/S_2],\GG_m)_{n} \times \mu_2 \stackrel{p_2}{\ra} \mu_2.$$ In either case, from (\ref{longseqwithchar}) we get a well-defined  map,  
\begin{equation}
H^2_{\acute{E}t}([X^{\times 2}/S_2],\GG_m)_{n} \ra \mu_2
\end{equation}
(which is trivial in the second case). Hence in both cases,  we get the following exact sequence
\begin{equation}
 0  \ra H^2_{\acute{E}t}(X^{(2)},\GG_m)_{2^m} \ra \ker(H^2_{\acute{E}t}([X^{\times 2}/S_2],\GG_m)_{2^m} \ra \mu_2) \ra H^1(X, \mu_2)
\end{equation}
 Since this hold for any $m \geq 1$ so we get (\ref{2*}).

\end{proof}

 For any two abelian varieties $A_1,A_2$ over a field $k$ of arbitrary characteristic, let $Hom(A_1,A_2)$ denote the group of homomorphisms of $A_1$ into $A_2$ (\cite[\S 19]{mumfordablvar}). By \cite[Corollary 1, \S 19]{mumfordablvar},  for a certain $\rho \leq 4 dim(A_1) dim(A_2)$ we have  
\begin{equation} \label{freefg} Hom(A_1,A_2) \simeq \ZZ^{\rho}.
\end{equation} 
Thus in any characteristic $Hom(A_1,A_2)$ is a finitely generated free abelian group.
\begin{prop} \label{EFvanish} Let $n \in \NN$. Let $S_2$ act on $\Pic(X^{\times 2})$ by action $\sigma$. Let $\Pic(X^{\times 2})n$ denote the elements of $n$-torsion. Then 
\begin{equation} H^i(S_2,\Pic(X^{\times 2})_n)=0 \quad \text{for} \quad i \in \{1,2\}.
\end{equation}
\end{prop}
\begin{proof} 
By \cite[Theorem 7.1]{maclane}, we have
\begin{eqnarray} \label{EF}  H^1(S_2,\Pic(X^{\times 2})_n)= \frac{\{L \in \Pic(X^{\times 2})_n | \sigma(L) \simeq L^{-1} \}}{\{ \sigma(L) \otimes L^{-1}|L \in \Pic(X^{\times 2})_n \}} \\ H^2(S_2,\Pic(X^{\times 2})_n)= \frac{\Pic(X^{\times 2})_n^{S_2}}{\{\sigma(L) \otimes L |L \in \Pic(X^{\times 2})_n \}} 
\end{eqnarray}

Recall that $\Pic(X^{\times 2})=\Pic(X)^{\times 2} \times Hom(Alb(X),\Pic^0(X))$ where $\Pic^0(X)$ is the Picard variety of $X$ parametrizing algebraically trivial line bundles on $X$ and $Alb(X)$ is the Albanese variety.

%Consider the exact sequence $0 \ra \Pic^0(X)[2] \ra \Pic^0(X) \stackrel{2}{\ra} \Pic^0(X) \ra 0$. Applying $Hom(Alb(X),?)$ since  $Ext^1(Alb(X),\Pic^0(X)[2])$ is a finite group scheme, so it follows that the group scheme $\Hom(Alb(X),\Pic^0(X))$ is divisible by two i.e for each $k$-rational point $\psi$ there exists another point $\psi'$ such that $\psi=2\psi'$.

Now $\Pic(X^{\times 2})_n=\Pic(X)^{\times 2}_n$ since by (\ref{freefg}) the abelian group $Hom(Alb(X),\Pic^0(X))$ is torsion-free.

 Since $Alb(X)$ and $\Pic^0(X)$ are dual to each other, so for a homomorphism $\psi: Alb(X) \ra \Pic^0(X)$ of abelian schemes, let $\psi^{\vee}: Alb(X) \ra \Pic^0(X)$ denote the dual homomorphism. Therefore writing $L \in \Pic(X^{\times 2})$ as $(L_1,L_2,\psi) \in \Pic(X)^2 \times Hom(Alb(X),\Pic^0(X))$ we see that 
\begin{eqnarray}
\sigma(L_1,L_2,\psi)=(L_2,L_1,\psi^{\vee}) \\
(L_1,L_2,\psi)^{-1}=(L_1^{-1},L_2^{-1},-\psi).
\end{eqnarray}
So the condition $\sigma(L) \simeq L^{-1}$ means that $L_1=L_2^{-1}$ and $\psi^{\vee}=-  \psi$. Further if $L \in \Pic(X^{\times 2})_n$ then $\psi=0$.
Putting these relations in (\ref{EF}), it follows that $H^i(S_2,\Pic(X^{\times 2})_n)=0$ for $i \in \{1,2\}$.

\end{proof}

%We now briefly recall abstract kernels from \cite[Chapter IV,\S 8]{maclane}. Any group extension $$E: 0 \ra G \stackrel{x}{\ra} B \stackrel{\sigma}{\ra} \Pi \ra 1$$ determines an action of $\theta: B \ra \Aut(G)$. Hence we get an induced homomorphism $\psi: \Pi \ra Out(G)$. The triple $(\Pi,G,\psi)$ is called an abstract kernel. To each abstract kernel corresponds an element in $H^3(\Pi, Z(G))$ called the obstruction class of the abstract kernel. An abstract kernel $(\Pi,G,\psi)$ comes from an extension of $\Pi$ by $G$ if and only if its obstruction class is zero \cite[\S 2, (3)]{wu}.

For any $m \geq 1$, let $n=2^m$. Consider the short exact sequence
\begin{equation}
0 \ra NS(X^{\times 2}) \otimes \ZZ/n \ra H^2(X^{\times 2},\mu_n) \ra Br'(X^{\times 2})_n \ra 0.
\end{equation}
Taking $S_2$ invariants, using the long exact sequence, we set 
\begin{equation} 
K_{2^m} := \ker(Br'(X^{\times 2})_{2^m}^{S_2} \ra H^1(S_2, NS(X^{\times 2}) \otimes \ZZ/2^m))
\end{equation}

\begin{thm} \label{allexcept2tor} 
%Let $BS_2=[\Spec(k)/S_2]$. Consider the spectral sequence for $\pi:  [X^{\times 2}/S_2] \ra BS_2$ with values in $\GG_m$ on the big \'etale site of these stacks\begin{equation}H^p_{\acute{E}t}(BS_2,R^q \pi_* \GG_m) \implies H^{p+q}_{\acute{E}t}([X^{\times 2}/S_2],\GG_m).\end{equation} We have $E^{3,0}_2=\mu_2$, $d^{1,1}_2: E \ra E^{3,0}_2$ and $d^{0,2}_2: Br(X^{\times 2})^{S_2} \ra F$. Further we have the exact sequence \begin{equation} 1 \ra \ker(d^{1,1}_2) \ra Br'([X^{\times 2}/S_2])  \ra  \ker(d^{0,2}_3: \ker(d^{0,2}_2) \ra E^{3,0}_3) \ra 1.\end{equation} 
The natural morphism \begin{equation} \label{oddisoinv}
Br'([X^{\times 2}/S_2])_{\ell^*} \ra Br(X^{\times 2})_{\ell^*}^{S_2}
\end{equation}
induces an isomorphism for $\ell$ odd. 
For $\ell=2$, we have the exact sequence
\begin{equation} \label{2mef}
0 \ra Br'([X^{\times 2}/S_2])_{2^m} \ra K_{2^m}  \ra \mu_2
\end{equation}
\end{thm}

\begin{proof} Let $n=\ell^m$ for $m \geq 1$. Consider the Leray spectral sequence for $\pi: [X^{\times 2}/S_2] \ra BS_2$ with values in $\mu_n$ in the big \'etale site of these stacks
\begin{equation}
H^p_{\acute{E}t}(BS_2,R^q \pi_* \mu_n) \implies H^{p+q}_{\acute{E}t}([X^{\times 2}/S_2],\mu_n).
\end{equation} 
By the proper base-change theorem in \'etale cohomology \cite[Expos\'e XII,XIII]{SGA4} \cite[(11.3.1)]{tamme}, the fibers of $R^n \pi_* \mu_n$ over $\Spec(k) \ra [\Spec(k)/S_2]=BS_2$ identify with $H^n(X^{\times 2}, \mu_n)$. So for $n=\{0,1,2 \}$ these are isomorphic to $\mu_n, \Pic(X^{\times 2})_n$ and $ H^2(X^{\times 2},\mu_n)$ respectively.  They have action induced by the $S_2$ action on $X^{\times 2}$. Since a sheaf on $BS_2$ is a sheaf on $\Spec(k)$ together with a $S_2$-action, so this describes the sheaves $R^n \pi_* \mu_n$ on $BS_2$ completely.

Thus $ E^{2,0}_2=H^2_{\acute{E}t}(BS_2,R^0 \pi_* \mu_n)=H^2(S_2,H^0(X^{\times 2},\mu_n))=H^2(S_2,\mu_n)$. So by \cite[Theorem 7.1]{maclane}, we have 
\begin{eqnarray}
E^{2,0}_2 = \mu_2 \quad \text{for} \quad \ell=2 \\  0 \quad \text{for} \quad \ell \neq 2.
\end{eqnarray}
Similarly 
\begin{equation} \label{e112=0}
E^{1,1}_2=H^1_{\acute{E}t}(BS_2,R^1 \pi_* \mu_n)=H^1(S_2,\Pic(X^{\times 2})_n)\stackrel{Prop \ref{EFvanish}}{=} 0.
\end{equation}
Finally
\begin{equation}
E^{0,2}_2=H^0_{\acute{E}t}(BS_2,R^2 \pi_* \mu_n)= H^0(S_2,H^2(X^{\times 2}, \mu_n))= H^2(X^{\times 2},\mu_n)^{S_2}.
\end{equation}
Also
\begin{equation}
E^{2,1}_2=H^2_{\acute{E}t}(BS_2,R^1 \pi_* \mu_n)=H^2(S_2,\Pic(X^{\times 2})_n)\stackrel{Prop \ref{EFvanish}}{=} 0.
\end{equation}
Thus 
\begin{equation} \label{e022e023} E^{0,2}_3=E^{0,2}_2
\end{equation} We want to analyse $d^{0,2}_3:E^{0,2}_3 \ra E^{3,0}_3$. Now $$E^{3,0}_2=H^3(BS_2,H^0(X^{\times 2}, \mu_n))=H^3(S_2,\mu_n).$$ So by \cite[Theorem 7.1]{maclane}, we get
\begin{eqnarray}
E^{3,0}_2= \mu_2 \quad  \quad \ell=2 \\= 0 \quad  \quad \ell \neq 2.
\end{eqnarray}
Thus for $\ell$ odd, we have $E^{0,2}_4=E^{0,2}_3$. Thus we have
\begin{equation} \label{isomun}
H^2_{\acute{E}t}([X^{\times 2}/S_2],\mu_n)= H^2_{\acute{E}t}(X^{\times 2},\mu_n)^{S_2} \quad \text{for} \quad \ell \neq 2.
\end{equation}
We will now compute the cohomological Brauer group of the stack $[X^{\times 2}/S_2]$. We begin with two preparations which hold both for $\ell \neq 2$ and $\ell=2$ cases. In the proof of Proposition \ref{A1}, we showed that the exact sequence (\ref{picwithlin}):
\begin{equation*}  0 \ra \mu_2 \ra \Pic_{S_2}(X^{\times 2}) \ra \Pic(X^{\times 2})^{S_2} \ra 0
\end{equation*}
is split exact. Thus algebraically trivial line bundles in $\Pic_{S_2}(X^{\times 2})$ and $\Pic(X^{\times 2})^{S_2}$ are isomorphic. Further by (\ref{stackwithlin}) we have $\Pic_{S_2}(X^{\times 2})=\Pic([X^{\times 2}/S_2])$. Thus we get the split exact sequence
\begin{equation}
0 \ra \mu_2 \ra NS([X^{\times 2}/S_2]) \ra NS(X^{\times 2})^{S_2} \ra 0.
\end{equation}
Upon tensoring with $\ZZ/n$ we get 
\begin{eqnarray} 
\label{nsstacks2invnot2} NS([X^{\times 2}/S_2]) \otimes \ZZ/n \simeq NS(X^{\times 2})^{S_2} \otimes \ZZ/n \quad \text{for} \quad \ell \neq 2 \\
\label{nsstacksinv2} NS([X^{\times 2}/S_2]) \otimes \ZZ/n \simeq NS(X^{\times 2})^{S_2} \otimes \ZZ/n \oplus \mu_2 \quad \text{for} \quad \ell=2
\end{eqnarray}

Now we come to the second preparation. By the Kummer exact sequence we get
\begin{eqnarray} \label{kummerseqgen}
0 \ra NS(X^{\times 2}) \otimes \ZZ/n \ra H^2(X^{\times 2}, \mu_n) \ra Br'(X^{\times 2})_n \ra 0 \\
 \label{kummerforstack} \qquad \qquad
0 \ra NS([X^{\times 2}/S_2]) \otimes \ZZ/n \ra H^2([X^{\times 2}/S_2],\mu_n) \ra Br'([X^{\times 2}/S_2])_n \ra 0.
\end{eqnarray}

Now we start computing the cohomological Brauer group. We want to take invariants under $S_2$ action. We first take the case $\ell \neq 2$.
 Since $$H^1(S_2,NS(X^{\times 2})  \otimes \ZZ/n)=0$$
because $NS(X^{\times 2}) \otimes \ZZ/n$ has no elements of two torsion, so we get
\begin{equation} \label{s2kummer}
0 \ra NS(X^{\times 2})^{S_2} \otimes \ZZ/n \ra H^2(X^{\times 2}, \mu_n)^{S_2} \ra Br' (X^{\times 2})_n^{S_2} \ra 0.
\end{equation}

 Quotienting (\ref{isomun}) 
by (\ref{nsstacks2invnot2}), using (\ref{s2kummer}) and (\ref{kummerforstack}) the induced map
\begin{equation}
Br'([X^{\times 2}/S_2])_n \ra Br'(X^{\times 2})_n^{S_2} \quad \text{for} \quad \ell \neq 2.
\end{equation}
becomes an isomorphism. This establishes (\ref{oddisoinv}).
%In the proof below, each of the sheaves $R^q \pi_* \GG_m$ on the \'etale site of $BS_2$ will be considered as sheaves on the \'etale site of $\Spec(k)$ together with a $S_2$-action.
Now let $\ell=2$. We have
\begin{equation} \label{d023}
d^{0,2}_3: H^2(X^{\times 2}, \mu_n)^{S_2}= E^{0,2}_2 \stackrel{(\ref{e022e023})}{=} E^{0,2}_3 \ra E^{3,0}_3 = \mu_2.
\end{equation}
We have 
\begin{equation} \label{e02infty} E^{0,2}_\infty= E^{0,2}_4=\ker(d^{0,2}_3).
\end{equation}
Now $E^{0,1}_2=H^0(BS_2,R^1 \pi_* \mu_n)=\Pic(X^{\times 2})_n^{S_2}=\Pic(X)_n$. Further the map 
\begin{equation}
d^{0,1}_2: \Pic(X^{\times 2})_n^{S_2} \ra H^2(BS_2,\mu_n)
\end{equation}
associates to a $S_2$-invariant line bundle the extension class of its theta group scheme \cite[\S 23]{mumfordablvar}. Namely, given a line bundle $L \in \Pic(X^{\times 2})_n^{S_2}$ consider the group scheme $$Mum(L)=\{ (g,\theta) | \theta: L \ra g^* L \quad \text{is an isomorphism of $\mu_n$ bundles} \}.$$ The map $d^{0,1}_2$ associates 
\begin{equation}
L \ms [1 \ra \mu_n \ra Mum(L) \ra S_2 \ra 0] \in H^2(BS_2,\mu_n).
\end{equation}
(cf \cite{me} for a similar $d^{0,1}_2$ computation result. Although the contexts are different in details but the formal context of $d^{0,1}_2$ computation is the same.)

 By the proof of Proposition \ref{EFvanish}, we have $$\Pic(X^{\times 2})_n^{S_2}=\{ (L,L,0) | L \in \Pic(X)_n \}= \Pic(X)_n.$$ So the $S_2$ action on these invariant line bundles is the trivial action. For $L \in \Pic(X)_n^{S_2}$ and $g \in S_2$ we have a canonical isomorphism $L \ra g^*L$ given by identity. So the extension class is also trivial for all line bundles. Thus we get $
d^{0,1}_2=0$. Thus 
\begin{equation} \label{e20infty} E^{2,0}_\infty= E^{2,0}_3=E^{2,0}_2=\mu_2.
\end{equation} Thus since $E^{1,1}_\infty=E^{1,1}_2=0$ by (\ref{e112=0}), so by (\ref{e20infty}) and (\ref{e02infty}) we get
\begin{equation} \label{mu2irritating}
1 \ra \mu_2 \ra H^2([X^{\times 2}/S_2],\mu_n) \ra ker(d^{0,2}_3: H^2(X^{\times 2}, \mu_n)^{S_2} \ra \mu_2) \ra 0.
\end{equation}
Now we want to pass from $\mu_n$ coefficients to the cohomological Brauer group.  Under the composition $H^2([X^{\times 2}/S_2],\mu_n) \ra ker \hra H^2(X^{\times 2}, \mu_n)^{S_2}$ the image of $NS([X^{\times 2}/S_2]) \otimes \ZZ/n$ is $NS(X^{\times 2}) \otimes \ZZ/n$. So by (\ref{nsstacksinv2}), it has kernel $\mu_2$. So the $\mu_2$ of (\ref{mu2irritating}) and (\ref{nsstacksinv2}) may be identified. Taking quotient of (\ref{mu2irritating}) by (\ref{nsstacksinv2}),  by (\ref{kummerforstack}) we get an isomorphism
\begin{equation}
Br'([X^{\times 2}/S_2])_n \ra ker( \frac{H^2(X^{\times 2},\mu_n)^{S_2}}{NS(X^{\times 2})^{S_2} \otimes \ZZ/n} \ra \mu_2)_n
\end{equation}
Now  we only want to express the second term differently.
Taking the long exact sequence associated to taking $S_2$-invariants in (\ref{kummerseqgen}) we get $$ NS(X^{\times 2})^{S_2} \otimes \ZZ/n \ra H^2(X^{\times 2}, \mu_n)^{S_2} \ra Br'(X^{\times 2})_n^{S_2} \ra H^1(S_2, NS(X^{\times 2}) \otimes \ZZ/n) $$ Therefore if we set
$K_{2^m}:=\ker(Br'(X^{\times 2})_n^{S_2} \ra H^1(S_2, NS(X^{\times 2}) \otimes \ZZ/n))$
we get the exact sequence
\begin{equation}
0 \ra Br'([X^{\times 2}/S_2])_n \ra  K_{2^m} \ra \mu_2
\end{equation}
This establishes (\ref{2mef}).
\end{proof}

\section{The case $k=\CC$}
\subsection{Generalities}
 For a complex analytic space $M$ the analytic Brauer group is denoted $Br_{an}(M)$ and the  cohomological Brauer group is denoted $Br_{an}'(M)$. We have $Br_{an}'(M)=H^2(M,\cO^*_M)_{tor}$,  where $\cO^*_M$ denotes the sheaf of holomorphic functions with values in $\CC \setminus \{0\}$. We now recall a Brauer sequence due to S.Schroer. It provides the  fundamental framework in terms of which the results of this section will be expressed. 
\begin{prop} \label{sch} \cite[Prop 1.1]{schroer} Let $M$ be any complex-analytic space.  Then we have the exact sequence
\begin{equation}
0 \ra H^2(M,\ZZ)/NS(M) \otimes \QQ/\ZZ \ra Br'_{an}(M) \ra H^3(M,\ZZ)_{tor} \ra 0.
\end{equation}
\end{prop}
Recall that given any scheme $Y$ of finite type over $\CC$, we can associate to it a complex analytic space \cite[Appendix B, page 439]{hart}. If $Y$ is compact, then by \cite[Prop 1.3]{schroer} and \cite[Prop 1.4]{schroer} the natural morphism $Br'(Y) \ra Br_{an}'(Y)$ and $Br(Y) \ra Br_{an}(Y)$ are isomorphisms. In this paper, we will apply these results to the cases $Y$ equals $X^{[2]}$ and $X^{(2)}$ whose underlying complex analytic spaces are compact. Further we will refer to $H^3(M,\ZZ)_{tor}$ as the finite torsion subgroup part (or sometimes abusively as simply torsion part) and to $H^2(M,\ZZ)/NS(M) \otimes \QQ/\ZZ$ as the divisible subgroup part.
\subsection{Comparison of analytic Brauer groups}
Our main aim is to prove the following  theorem
\begin{thm}  \label{mtcomplex} Over complex numbers, the morphism $\omega_2:X^{[2]} \ra X^{(2)}$ induces an isomorphism on analytic cohomological Brauer groups.
\end{thm}
 Instead of the \'etale topology on schemes, we will use complex topology.

Let us recall some facts about proper base change theorem in topology from \cite[Chapter 17]{milneLEC}. A continous map is called proper if it is universally closed. Let $\pi: X \ra S$ be a continous map and $\cF$ a sheaf on $X$. For any $s \in S$, we set
\begin{equation}
(R^r \pi_* \cF)_s = \varinjlim H^r(\pi^{-1}(V),\cF),
\end{equation}
where the limit is taken over open neighbourhoods $V$ of $s \in S$. We quote
\begin{thm} \cite[Theorem 17.2,17.3]{milneLEC} Let $\pi: X \ra S$ be a proper map where $S$ is a locally compact topological space. For any sheaf $\cF$ on $X$ and $s \in S$, we have
\begin{equation}
(R^r \pi_* \cF)_s = H^r(X_s,\cF)
\end{equation}
where $X_s = \pi^{-1}(s)$. Consider the cartesian square
\begin{equation}
\xymatrix{
X' \ar[r]_{f'} \ar[d]^{\pi'} & X \ar[d]^{\pi} \\
T \ar[r]^f & S 
}
\end{equation}
Then the canonical base-change homomorphism $$f^* (R^r \pi_* \cF) \ra R^r \pi'_* (f^{'*} \cF)$$ is an isomorphism for any $r$.
\end{thm}

Now we return to the diagram (\ref{diag2}).
\begin{prop} \label{onemoreh2} We have $H^0(X^{(2)},R^2 \omega_{2*} \ZZ)=H^0(X,R^2 \pi_* \ZZ) \simeq \ZZ$ and a short exact sequence 
\begin{equation} \label{eexactness} 0 \ra H^2(X^{(2)},\ZZ) \ra H^2(X^{[2]},\ZZ) \ra H^0(X^{(2)},R^2 \omega_{2*} \ZZ) \ra 0.
\end{equation}
\end{prop}
\begin{proof}Now $\omega_{2*}\ZZ$ identifies with the sheaf associated to $\ZZ$. 
 We consider the commutative diagram (\ref{diag2}) by taking underlying complex analytic spaces. In the topological category the map $\omega_2$ is proper. Therefore the canonical base-change homomorphism
\begin{equation} \label{basechangehom}
 i^* R^q \omega_{2,*} \ZZ \ra R^q \pi_* j^* \ZZ
\end{equation}
is an isomorphism for any $q$ by \cite[Theorem 17.3]{milneLEC}. By \cite[Theorem 17.2]{milneLEC}, for any sheaf $\cF$ on $X^{[2]}$, we have $(R^q \omega_{2,*} \cF)_{s}=H^q(\omega_2^{-1}(s),\cF)$. Let $U:= X^{(2)} \setminus X$. Since $\omega_2^{-1} U \ra U$ is an isomorphism, so it follows that the sheaf $ R^q \omega_{2,*} \ZZ$ is supported on $i(X)$ for $q \geq 1$. So 
the morphism obtained from (\ref{basechangehom}) after adjunction 
\begin{equation} \label{romega2}
R^q \omega_{2*} \ZZ \ra i_* R^q \pi_{*} j^* \ZZ
\end{equation}
is an isomorphism for any $q \geq 1$.

 Let us consider the Leray-Serre spectral sequence for $\omega_2: X^{[2]} \ra X^{(2)}$
\begin{equation} \label{lerayforomega2} E^{p,q}_2=H^p(X^{(2)},R^q \omega_{2*} \ZZ) \implies E^{p+q}=H^*(X^{[2]},\ZZ)
\end{equation}  where we denote by $\ZZ$ the constant sheaf defined by the integers. Consider the associated $7$-term long-exact sequence (\cite[page 371, Cor 3.2]{merkur})

$$0 \ra E^{1,0}_2 \ra E^1_\infty \ra E^{0,1}_2 \ra E^{2,0}_2 \ra ker(E^2_\infty \ra E^{0,2}_2) \ra E^{1,1}_2$$

We have 
\begin{equation} \label{r1omega2z} R^1 \omega_{2*} \ZZ \stackrel{(\ref{romega2})}{=} i_* R^1 \pi_* j^* \ZZ=0,
\end{equation} and hence $E^{0,1}_2=E^{1,1}_2=0$. 
We get an isomorphism of groups
\begin{equation} \label{6term}
E^{2,0}_2 \ra ker(E^2_{\infty} \ra E^{0,2}_2).
\end{equation}
Since $E^2_\infty=H^2(X^{[2]},\ZZ)$, we deduce the following exact sequence
\begin{equation} \label{7term}
0 \ra H^2(X^{(2)},\omega_{2,*} \ZZ) \ra H^2(X^{[2]},\ZZ) \ra H^0(X^{(2)},R^2 \omega_{2,*} \ZZ) 
\end{equation}

Consider the sheaf $R^2 \pi_{*} \ZZ$. Its fiber over $x \in X$ is  $H^2(\PP^1(\Omega_{X,x}),\ZZ)$. It identifies canonically with $\ZZ$ since $\PP^1_\CC$ is orientable. Further  since $\pi$ is an $\PP^1$-fibration, so on the overlaps the gluing functions preserve the orientation class of the fiber. Thus 
\begin{equation} \label{R2piZ} R^2 \pi_{*} \ZZ = \ul{\ZZ}
\end{equation} where $\ul{\ZZ}$ is  the constant sheaf defined by integers $\ZZ$. Therefore 
\begin{equation} H^0(X,R^2 \pi_{*} \ZZ)=\ZZ.
\end{equation}

 Therefore by (\ref{romega2}), we have 
 \begin{equation} H^0(X^{(2)},R^2 \omega_{2*} \ZZ) = H^0(X,R^2 \pi_{*} \ZZ)= \ZZ.
 \end{equation}
  
By (\ref{7term}) we have the exact sequence
\begin{equation} \label{gamma}
0 \ra H^2(X^{(2)},\ZZ) \ra H^2(X^{[2]},\ZZ) \stackrel{\gamma}{\ra} H^0(X^{(2)},R^2 \omega_{2*} \ZZ) \simeq \ZZ
\end{equation}
Consider $\PP(\Omega_X) \hra X^{[2]}$ as a divisor in a smooth scheme. Since $\omega_{2}(\PP(\Omega_X))=X$ which is a codimension two subscheme of $X^{(2)}$, so the class $[\PP(\Omega_X)] \in H^2(X^{[2]},\ZZ)$ of $\PP(\Omega_X)$   does not come from $H^2(X^{(2)},\ZZ)$. Further under the composition $$H^2(X^{[2]},\ZZ) \ra H^0(X^{(2)},R^2 \omega_{2,*} \ZZ)=H^0(X,R^2 \pi_{*} \ZZ)=\ZZ$$
the  image of $[\PP(\Omega_X)]$ maps to the generator. 
This shows the surjectivity of $\gamma$. So (\ref{gamma}) is exact also on the right.
In other words, we have shown that (\ref{gamma}), or equivalently (\ref{eexactness}), is  an exact sequence. 
\end{proof}

\begin{prop} \label{H_3tor} The edge morphism $E^{3,0}_2 \ra E^3$ of the Leray spectral sequence for $\omega_2$ (cf (\ref{lerayforomega2}))  induces an isomorphism on torsion subgroups. In other words, the following edge map is an isomorphism
\begin{equation} H^3(X^{(2)}, \ZZ)_{tor}=H^3(X^{(2)},\omega_{2*}\ZZ)_{tor} \stackrel{edge}{\lra}  H^3(X^{[2]},\ZZ)_{tor}
\end{equation}
\end{prop}
\begin{proof}  
 We will first show that each one of $E_\infty^{0,3},E_\infty^{1,2}$, $E_\infty^{2,1}$ are torsion-free and then show that the torsion in  $E^{3,0}_2=H^3(X^{(2)},\omega_{2*}\ZZ)$ survives till $E_\infty^{3,0}$.

Now $E_2^{0,3}$ is  zero because the sheaf $R^q \omega_{2*} \ZZ=0$ for $q \geq 3$. 

Now $E^{1,2}_2=H^1(X^{(2)},R^2 \omega_{2*} \ZZ) \stackrel{(\ref{romega2})}{ =} H^1(X,R^2 \pi_* \ZZ) \stackrel{(\ref{R2piZ})}{=} H^1(X,\ZZ),$
which is always torsion-free by the universal coefficient theorem.

 Similarly $R^1 \omega_{2*} \ZZ=0$ by (\ref{r1omega2z}). So since $E^{-1,3}_2 \ra E^{1,2}_2 \ra E^{3,1}_2$, it follows that $$E^{1,2}_2 = E^{1,2}_3.$$ So $E^{1,2}_3$ is torsion-free. Now by $E^{-2,4}_3 \ra E^{1,2}_3 \ra E^{4,0}_3$ it follows that $$E^{1,2}_4 \hra E^{1,2}_3.$$ So $E^{1,2}_4$ is torsion-free. By $E^{-3,5}_4 \ra E^{1,2}_4 \ra E^{5,-1}_4$ it follows that $$E^{1,2}_\infty=E^{1,2}_4.$$ So $E^{1,2}_\infty$ is torsion-free.

Now $E^{2,1}_2=H^2(X,R^1 \omega_{2,*} \ZZ)=0$ since $R^1 \omega_{2*} \ZZ=0$ by (\ref{r1omega2z}).

Lastly we consider $E^{3,0}_2$. By $E^{1,1}_2 \ra E^{3,0}_2 \ra E^{5,-1}_2$, it follows that 
\begin{equation} \label{lastly} E^{3,0}_2=E^{3,0}_3.
\end{equation} Consider the differentials at level three
\begin{equation}  \label{d}
E^{-3,4}_3 \ra E^{0,2}_3 \stackrel{d}{\ra} E^{3,0}_3 \ra E^{6,-2}_3 
\end{equation} 
We claim that the differential $d: E^{0,2}_3 \ra E^{3,0}_3$ is zero. Now (\ref{eexactness}) can be restated as  the exactness of
\begin{equation} \label{eeexactness}
0 \ra E^{2,0}_2 \ra E^2_\infty \ra E^{0,2}_2 \ra 0.
\end{equation}
In general one has 
$
\xymatrix{
E^2_\infty \ar@{->>}[r] & E^{0,2}_\infty \ar@{^{(}->}[r] & \cdots E^{0,2}_4 \ar@{^{(}->}[r] & E^{0,2}_3 \ar@{^{(}->}[r] & E^{0,2}_2
}$.

Since both arrows above are edge morphisms, so by the surjectivity of the second edge morphism we have $$E^{0,2}_2=E^{0,2}_\infty.$$ 

%On the other hand, $ E^{0,2}_2=H^0(X^{(2)},R^2 \omega_{2,*} \ZZ)$ $$ \stackrel{(\ref{romega2})}{=} H^0(X^{(2)},i_* R^2 \pi_{*} j^*\ZZ)=H^0(X,R^2 \pi_* \ZZ) \stackrel{(\ref{R2piZ})}{=} H^0(X,\ul{\ZZ}) \simeq \ZZ.$$
Thus $E^{0,2}_3=E^{0,2}_4$. So $d$ in (\ref{d}) is zero.  So $E^{3,0}_4=E^{3,0}_3$. Thus we have 
$$E^{3,0}_\infty=E^{3,0}_4=E^{3,0}_3 \stackrel{(\ref{lastly})}{=}E^{3,0}_2=H^3(X^{(2)},\omega_{2*} \ZZ).$$ So it follows that the edge morphism $H^3(X^{(2)}, \omega_{2*} \ZZ) \hra H^3(X^{[2]},\ZZ)$
is injective. 
By the torsion-freeness of all other groups at infinity in the spectral sequence  we have $H^3(X^{(2)},\omega_{2,*} \ZZ)_{tor}=H^3(X^{[2]},\ZZ)_{tor}$. 

\end{proof}

Therefore by Propositions \ref{A1} and \ref{onemoreh2}, it follows that the  morphism $\omega_2$ induces an isomorphism 
\begin{equation} \label{onemore} H^2(X^{(2)},\ZZ)/NS(X^{(2)}) \ra H^2(X^{[2]},\ZZ)/NS(X^{[2]}).
\end{equation}

\begin{proof}[Proof of Theorem \ref{mtcomplex}] 
%The Hilbert scheme $X^{[d]}$ is a smooth by \cite[Fogarty 68]{fogarty68}. So $Br'(X^{[2]})=Br(X^{[2]})$. 
Combining (\ref{onemore})  with
Proposition \ref{H_3tor}, by the sequence due to Schroer (cf Proposition \ref{sch}) for $\omega_2:X^{[2]} \ra X^{(2)}$  we get our result.
\end{proof}

\begin{Cor} The group $H^2_{an}(X^{(2)},\GG_m)$ is torsion.  \end{Cor}\begin{proof} For the complex topology, consider the Leray spectral sequence for the map $\omega_2: X^{[2]} \ra X^{(2)}$:\begin{equation} E^{p,q}_2=H^p(X^{(2)},R^q \omega_{2,*} \ul{\GG_m}) \implies H^{p+q}(X^{[2]}, \ul{\GG_m})\end{equation}where $\ul{\GG_m}$ denotes the sheaf defined by $\GG_m$. Consider the $7$-term long sequence\begin{equation}  0 \ra E^{1,0}_2 \ra E^1 \ra E^{0,1}_2 \ra E^{2,0}_2 \ra ker(E^2 \ra E^{0,2}_2) \ra E^{1,1}_2 \end{equation} Now $\omega_{2,*} \ul{\GG_m} = \ul{\GG_m}$. Now the part $0 \ra E^{1,0}_2 \ra E^1 \ra E^{0,1}_2$ identifies with \begin{equation}0 \ra H^1(X^{(2)},\GG_m) \ra H^1(X^{[2]},\GG_m) \ra H^0(X^{(2)}, R^1 \omega_{2,*} \GG_m)\end{equation} It is surjective because $\PP(\Omega_X) \hra X^{[2]}$ defines a divisor whose associated line bundle maps to the generator of $H^0(X^{(2)}, R^1 \omega_{2,*} \GG_m)$. Therefore $$H^2(X^{(2)},\GG_m)=E^{2,0}_2 \hra E^2=H^2(X^{[2]},\GG_m)$$

Since $X^{[2]}$ is smooth, so $H^2_{\acute{e}t}(X^{[2]},\GG_m)$ is torsion by a theorem of Grothendieck. But this equals $H^2_{an}(X^{[2]},\GG_m)$ by \cite[Prop 1.3,1.4]{schroer}. So $H^2_{an}(X^{(2)},\GG_m)$ is also torsion. 

%Further every class is represented by $\mu_n$ coefficients. By \cite[Expos\'e 16, Thm 4.1]{SGA4} which holds for complex but not necessarily for smooth schemes, we have \begin{equation} H^2_{an}(X^{(2)},\mu_n)=H^2_{\acute{e}t}(X^{(2)},\mu_n). \end{equation} Now by the Kummer exact sequence we have \begin{equation} 0 \ra NS(X^{(2)}) \otimes \ZZ/n \ra H^2(X^{(2)},\mu_n) \ra Br(X^{(2)})_{n} \ra 0 \end{equation} both for the analytic and the \'etale topologies. Further by GAGA analytic and algebraic line bundles on $X^{(2)}$ are equivalent. So $H^2_{an}(X^{(2)},\GG_m)=Br'(X^{(2)})$. 
\end{proof}

So by Theorem \ref{mtcomplex}, $H^2_{an}(X^{(2)},\GG_m)=Br'_{an}(X^{(2)})$.

\section{Computations}
 Recall that if $Y$ and $Z$ are smooth proper schemes over $k$, then (cf \cite[\S 6]{fogarty})
\begin{equation} \label{lbonprod}
\Pic(Y \times Z) \simeq \Pic(Y) \times \Pic(Z) \times \Hom(Alb(Y),\Pic^0(Z))
\end{equation}
where $Alb(Y)$ is the Albanese variety of $Y$. Classically $\Hom(Alb(Y),\Pic^0(Z))$ is called the group of {\it divisorial classes} between $Y$ and $Z$. Following \cite{bostcharles}, we will denote it as $DC(Y,Z)$ . Fixing a closed point $y \in Y$ and $z \in Z$ we may realise $DC(Y,Z)$ as a subgroup of $\Pic(Y \times Z)$ as follows:
\begin{equation} DC(Y,Z)= \{ L \in \Pic(Y \times Z) | L|_{y \times Z} \simeq \cO_{y \times Z} \quad L|_{Y \times z} \simeq \cO_{Y \times z} \}
\end{equation}

We now need to quote several results from  \cite{sga7} or \cite{bostcharles}. Although in \cite{bostcharles}, these are stated for schemes over $\ol{\QQ}$, but as remarked in \cite[\S 3.2]{bostcharles}, suitably formulated these hold over an arbitrary base. Since we work over algebraically closed fields, this should not cause a problem.  By \cite[(3.3) page 14]{bostcharles}, we have 
\begin{equation}
NS(Y \times Z) \simeq NS(Y) \times NS(Z) \times DC(Y,Z).
\end{equation}
 Let $S_n$ denote the symmetric group on $n$ letters. Taking $Y=Z=X$, we get a $S_2$ action on $DC(X,X)$. Classically one calls the following
\begin{equation} \label{symmetriccllasses}
C_s(X):=Hom(A(X),\Pic^0(X))^{S_2} = DC(X,X)^{S_2}
\end{equation} the group of {\it symmetric divisorial classes} (cf \cite{fogarty}).  

By \cite[Prop 3.5]{bostcharles}, fixing a closed point  $x \in X$ we get isomorphisms
\begin{equation}
(alb_{X,x},alb_{X,x})^*: DC(Alb(X),Alb(X)) \stackrel{\simeq}{\ra} DC(X,X)
\end{equation}
By \cite[(2.9.6.2)]{sga7} or \cite[Prop 3.6]{bostcharles} we have isomorphisms
\begin{equation}
DC(Alb(X),Alb(X)) \simeq Hom_k(Alb(X),Alb(X)^{\vee}).
\end{equation}
If we define $Hom_k^{sym}(A,A^{\vee})=\{\psi \in Hom_k(A,A^{\vee}) | \psi^{\vee}=\psi \}$, then by \cite[Expos\'e VIII.4 (4.14.1)]{sga7} or \cite[Prop 3.7]{bostcharles} we have isomorphism
\begin{equation}
NS(Alb(X)) \simeq Hom_k^{sym}(Alb(X),Alb(X)^{\vee}).
\end{equation}
Therefore by the preceeding four isomorphisms, we get
\begin{equation} \label{csxasnsalb}
NS(Alb(X))= C_s(X).
\end{equation}
%We have \begin{equation} \label{lbonsym2} \Pic(X^{(2)}) = \Pic(X) \times C_s(X).\end{equation}

\subsection{Computations for algebraically closed field $k$}
Let us cite the following result by Grothendieck.
\begin{prop} \label{grothendieck} \cite[Corollaire 3.4, page 82]{grothendieckbrauerii} Let $Z$ be a smooth geometrically integral variety over $k$. Let $\rho$ be the rank of $NS(Z)$. Let $b_2$ be the second $\ell$-adic Betti number of $Z$. Then the $\ell$-primary component $Br(Z)_{\ell^*} \subset Br(Z)$ is an extension of $H^3_{\acute{e}t}(Z,\ZZ(1))_{tor}$ by $(\QQ_{\ell}/\ZZ_{\ell})^{b_2- \rho}$. 
\end{prop}
Since the Weil conjectures have been proven by Deligne, so for $\ell \neq p$, we may refer to the $\ell$-adic Betti numbers as simply the Betti numbers of $Z$.

\begin{thm} \label{ansd=2}  Let $X$ be a smooth projective surface over an algebraically closed field $k$. Let $\ell \neq \{p,2\}$. Let  \begin{equation} \label{a}
a:= b_2+b_1^2- rank NS(X)- rank NS(Alb(X)) 
\end{equation} 
where $b_i$ are the $i$-th Betti numbers.
Let $l \neq \{2,p\}$ be a prime.
We have 
\begin{eqnarray*} Br(X^{[2]})_{\ell^*}=  (\QQ_{\ell}/\ZZ_{\ell})^{\oplus a} \oplus  H^3(X, \ZZ_{\ell})_{\ell^*} \oplus [H^1(X,\ZZ_{\ell}) \otimes H^2(X, \ZZ_{\ell})]_{\ell^*} \\ \oplus  Tor_1^{\ZZ_{\ell}}(H^1(X,\ZZ_{\ell}),H^3(X,\ZZ_{\ell}))_{\ell^*} \oplus Tor_1^{\ZZ_{\ell}}(H^2(X,\ZZ_{\ell}),H^2(X,\ZZ_{\ell}))_{\ell^*} 
\end{eqnarray*} 

Let $k=\CC$ and suppose $H_1(X)_{tor}$ has no two torsion element, then 
\begin{equation} \label{no2torc} Br(X^{[2]})=(\QQ/\ZZ)^a \oplus H_1(X)_{tor}^{\oplus (b_1+1)} \oplus Tor_1^{\ZZ}(H^2(X,\ZZ),H^2(X,\ZZ)) \end{equation}

\end{thm}
\begin{rem} For Catanese and Godeaux surfaces we have $\pi_1(X)=\ZZ/5$. Therefore $\pi_1^{\acute{e}t}(X)=\ZZ/5$. Thus by the relation $H^1_{\acute{e}t}(X,\ZZ_5)=\varprojlim Hom(\pi_1^{\acute{e}t}(X),\ZZ/5^m)$ (cf \cite{milneLEC}) we see that $H^1_{\acute{e}t}(X,\ZZ_5)=\ZZ/5$ is torsion, unlike $H^1(X,\ZZ)$ which is always torsion-free. On the other hand, by Poincar\'e duality (cf \cite{milneLEC}) $H^3_{\acute{e}t}(X,\ZZ_5)=\ZZ/5$ also. So we have $Tor_1^{\ZZ_5}(H^1(X,\ZZ_5),H^3(X,\ZZ_5))=\ZZ/5$. Thus this extra term in the $\ell$-adic case is not superflous.
\end{rem}
\begin{proof} Since $X$ is projective, so $X^{[2]}$ is projective. So by Gabber's theorem $Br(X^{[2]})=Br'(X^{[2]})$. By Theorem \ref{anychar} (or by Theorem \ref{mtcomplex} for $\CC$), $Br'(X^{(2)}) \ra Br'(X^{[2]})$ is an isomorphism. By Theorem \ref{stacktosym2}, $ H^2_{\acute{E}t}(X^{(2)},\GG_m)_{l^*} \ra H^2_{\acute{E}t}([X^{\times 2}/S_2],\GG_m)_{l^*}$ is an isomorphism. By Theorem \ref{allexcept2tor}, we have $Br'([X^{\times 2}/S_2])_{l^*} \ra Br(X^{\times 2})_{l^*}^{S_2}$ is an isomorphism.

Recall that by the Kunneth formula in \'etale cohomology we have
\begin{equation*}H^n(X \times Y,\ZZ_{\ell}) = \sum_{i+j =n} H^i(X,\ZZ_{\ell}) \otimes H^j(Y,\ZZ_{\ell}) \oplus \sum_{p + q=n+1} Tor_1^{\ZZ_{\ell}}(H^p(X,\ZZ_{\ell}),H^q(Y,\ZZ_{\ell}))
\end{equation*}

Let us record some of its consequences.
For $b \in [1,2]$, let $X_b$ denote the $b$-th factor in $X^{\times 2}$. By straightforward induction, we deduce
\begin{eqnarray*} H^1(X^{\times 2},  \QQ_{\ell}) & =&  \oplus_{1 \leq b \leq 2} H^1(X_b,\QQ_{\ell}), \\H^2(X^{\times 2}, \QQ_{\ell}) & = & \oplus_{1 \leq b \leq 2} H^2(X_b,\QQ_{\ell})  \oplus_{1 \leq b < c \leq 2} H^1(X_b,\QQ_{\ell}) \otimes H^1(X_c,\QQ_{\ell})
%\\ H^3(X^{\times d},\ZZ) &=& \oplus_{1 \leq a \leq d} H^3(X_a,\ZZ)  \oplus \\ && \oplus_{a \neq b} H^2(X_a,\ZZ) \otimes H^1(X_b,\ZZ) \oplus \\  &&  \oplus_{a < b < c} H^1(X_a,\ZZ) \otimes H^1(X_b,\ZZ) \otimes H^1(X_c,\ZZ) \oplus \\  && \oplus_{1 \leq a < b \leq n} Tor_1(H^2(X_a,\ZZ) ,H^2(X_b,\ZZ))
\end{eqnarray*}

The group $S_2$ permutes the indices $\{b|1 \leq b \leq 2 \}$ and $\{(b,c)|1 \leq b < c \leq 2 \}$. So taking $S_2$ invariants we get that the $\ZZ_{\ell}$ rank of $H^2(X^{\times 2},\ZZ_{\ell})^{S_2}$ is $b_2+b_1^2$.

Now the $\ZZ$ rank of $NS(X^{\times 2})^{S_2}$ is equal to that of $NS(X^{(2)})$.
By \cite[Fogarty (6)' page 678]{fogarty} we have 
\begin{equation} \label{rknsxd}
rank NS(X^{(2)})= rank NS(X) + rank C_s(X)
\end{equation}
where $C_s(X)$ is the group of symmetric divisor classes (cf \ref{symmetriccllasses}). Now rank of $NS(C_s(X))$ equals that of $NS(Alb(X))$ by (\ref{csxasnsalb}). So we get
\begin{equation}
rank NS(X^{\times 2})^{S_2} = rank NS(X) + rank NS(Alb(X)).
\end{equation}

Thus the integer $a$ (cf (\ref{a})) is determined. Now we derive $H^3(X^{[2]},\ZZ_{\ell})_{tor}$
by the Kunneth formula. Taking the torsion subgroup, we get $H^3(X^{\times 2}, \ZZ_{\ell})_{tor}=$
\begin{eqnarray*}  \sum_{i+j=3,0 \leq i,j \leq 3} [H^i(X,\ZZ_{\ell}) \otimes H^j(X, \ZZ_{\ell})]_{tor}   \oplus \sum_{p+q =4, 0 \leq p,q \leq 4} Tor_1(H^p(X,\ZZ_{\ell}),H^q(X,\ZZ_{\ell})) 
\end{eqnarray*}
Now the group $S_2$ permutes the indices $\{(i,j) | 0 \leq i, j \leq 3 \}$ and $\{(p,q)|0 \leq p,q \leq 4 \}$ between themselves.
Using the facts that $H^0(X,\ZZ_{\ell})=\ZZ_{\ell}$ and taking $S_2$ invariants, we may simplify this formula as follows. Combining the previous statements with Grothendieck's corollary (cf Proposition \ref{grothendieck}) we get our result which computes $Br(X^{[2]})_{\ell^*}$ for the case $\ell \neq 2,p$.

%By Proposition \ref{sch}, $Br'(X^{(2)})= (\QQ/\ZZ)^{a} \oplus H^3(X^{(2)},\ZZ)_{tor}$ where $a$ is the rank of $H^2(X^{(2)},\ZZ)/NS(X^{(2)})$. Now the rank of $H^2(X^{(2)},\ZZ)$ is $b_2+b_1^2$ by Proposition \ref{rankh2xd} and rank of $NS(X^{(2)})$ is the sum of ranks of $NS(X)$ and $NS(C_s(X))$ by (\ref{rknsxd}).  and $H_2(X^{(2)},\ZZ)_{tor}=H^3(X^{(2)},\ZZ)_{tor}$ by the universal coefficient theorem. Combining these with Proposition \ref{steenrod}, we get the first statement.

To deduce the second statement, we only need to determine  the Brauer group modulo its divisible subgroup. In other words, we want to calculate $H^3(X^{(2)},\ZZ)_{tor}$.  By the Kunneth formula we have$H^3(X^{\times 2}, \ZZ)=$
\begin{eqnarray*}  \sum_{i+j=3,0 \leq i,j \leq 3} [H^i(X,\ZZ) \otimes H^j(X, \ZZ)]   \oplus \sum_{p+q =4, 0 \leq p,q \leq 4} Tor_1(H^p(X,\ZZ),H^q(X,\ZZ)) 
\end{eqnarray*}

Now $H^0(X,\ZZ)=\ZZ$ and $H^1(X,\ZZ)=\ZZ^{b_1}$ are torsion-free and $H_1(X,\ZZ)_{tor} \simeq H^2(X,\ZZ)_{tor}$ by universal coefficient theorem and $H_1(X,\ZZ)_{tor} \simeq H^3(X,\ZZ)_{tor}$ by Poincar\'e duality. Thus we have $H^3(X^{ \times 2},\ZZ)^{S_2}_{tor}=$ $$[H^0(X,\ZZ) \otimes H^3(X,\ZZ) \oplus H^1(X,\ZZ) \otimes H^2(X,\ZZ)]_{tor} \oplus Tor_1^{\ZZ}(H^2(X,\ZZ),H^2(X,\ZZ)).$$
Thus we have
\begin{equation}
H^3(X^{\times 2},\ZZ)^{S_2}_{tor}=H_1(X,\ZZ)_{tor}^{\oplus (b_1 +1)} \oplus Tor_1^{\ZZ}(H^2(X,\ZZ),H^2(X,\ZZ)).
\end{equation} 
Thus if $H_1(X)_{tor}$ has no two torsion element, then $Br(X^{\times 2})^{S_2}$ does not have a non-divisible two torsion elements. Thus by Theorem \ref{allexcept2tor} it follows that $Br'([X^{\times 2}/S_2])$ is isomorphic to $Br(X^{\times 2})^{S_2}$. In particular, it does not have a non-divisible two torsion  element.

Combining the cases $l \neq 2$ and $l=2$ of Theorem \ref{stacktosym2} we obtain   the following single exact sequence
\begin{equation*}
0 \ra H^2_{\acute{E}t}(X^{(2)},\GG_m) \ra \ker(H^2_{\acute{E}t}([X^{\times 2}/S_2]),\GG_m) \ra \mu_2) \ra H^1(X, \mu_2)
\end{equation*}
Now divisible elements of $H^2_{\acute{E}t}([X^{\times 2}/S_2],\GG_m)$ must map to zero in $\mu_2$ and so must the remaining since they are not of two torsion. Therefore $H^2_{\acute{E}t}([X^{\times 2}/S_2],\GG_m)$ identifies with the kernel term. Similarly the map  $\ker \ra H^1(X,\mu_2)$ must be zero. Therefore we get the isomorphism 
\begin{equation}
 H^2_{\acute{E}t}(X^{(2)},\GG_m) \ra H^2_{\acute{E}t}([X^{\times 2}/S_2]),\GG_m) 
\end{equation}
Therefore upto divisible group part $H^2_{\acute{e}t}(X^{(2)},\GG_m)$ identifies with $H_1(X)_{tor}^{\oplus (b_1+1)} \oplus Tor_1^{\ZZ}(H^2(X,\ZZ),H^2(X,\ZZ))$.
\end{proof}

%The codimension of the complement of $ X^{[d]}_*$ in $X^{[d]}$ is equal to two.Since $X^{[d]}$ is smooth, so the restriction morphisms\begin{equation}Pic(X^{[d]}) \ra Pic(X^{[d]}_*), \quad \quad Br(X^{[d]}) \ra Br(X^{[d]}_*)\end{equation} are isomorphisms.

\begin{rem} \label{cctotaro} By \cite[Theorem 2.2]{totaro}, if the integral cohomology of $X$ has no $2$-torsion, then the integral cohomology of $X^{[2]}$ has no $2$-torsion either. In particular $H^3(X^{[2]},\ZZ)$ would have no two torsion. On the other hand,  by (\ref{no2torc}) of Theorem \ref{ansd=2} the integral cohomology of $X^{(2)}$ also does not have any element of two torsion. So our result is in agreement with a consequence of Totaro's result.
\end{rem}
\begin{rem} \label{computation} We could not follow the sketch of the description of the special polynomial algebra on \cite[page 257]{milgram} and therefore are unable to solve the second homology any further.
\end{rem} 

\begin{rem} \label{ccsteenrod} Over the complex numbers, we see that the Brauer group equals $H^3(X^{(2)},\ZZ)_{tor}$ modulo its divisible subgroup. Now $H^3(X^{(2)},\ZZ)_{tor}=H_2(X^{(2)},\ZZ)_{tor}$ by the universal coefficient theorem.
Let us choose a base point $x_0 \in X$. By a result of N.Steenrod \cite[(22.3),(22.5),(22.6)]{steenrod} (cf also \cite[page 2]{milgram} and \cite[\S 9, last para]{dold} for the following form) we have
\begin{equation}
H_2(X^{(2)},\ZZ)= H_2(X,x_0,\ZZ) \oplus H_2(X^{(2)},X,\ZZ).
\end{equation}
Since $H_2(X,x_0,\ZZ)=H_2(X,\ZZ)$, so we see that the $H^3(X^{(2)},\ZZ)_{tor}$ must contain a copy of $H_2(X,\ZZ)_{tor}=H^3(X,\ZZ)_{tor}$ as a direct summand by topological methods. Now $H^3(X,\ZZ)_{tor}$ is the Brauer group of the surface $X$ modulo its divisible subgroup. On the other hand, in Theorem \ref{ansd=2} for $\ell \neq 2$ we showed that the Brauer group contains $H^3(X,\ZZ_{\ell})_{\ell^*}$ using algebraic methods of \S \ref{cohbrgpsym2}. Thus our algebraic methods of \S \ref{cohbrgpsym2} agree with Steenrod's result for $\ell \neq 2$.
\end{rem}

\begin{Cor} Let $k=\CC$. If $H_1(X,\ZZ)$ is torsion-free, then $Br(X^{[2]})=(\QQ/\ZZ)^{a}$.
\end{Cor}
\begin{proof} Using universal coefficient theorem and Poincar\'e duality, this follows from (\ref{no2torc}).
\end{proof}
\begin{rem} After Theorem \ref{anychar} (or Theorem \ref{mtcomplex} for $k=\CC$), we have $Br'(X^{(2)}) \ra Br'(X^{[2]})$ is an isomorphism. By the universal coefficient theorem, $H^3(X^{(2)},\ZZ)_{tor}=H_2(X^{(2)},\ZZ)_{tor}$. Now $H_2(X^{(2)},\ZZ)_{tor}=0$ by \cite[Theorem 1.2]{totaro}. So our algebraic methods of \S \ref{cohbrgpsym2} agree with B.Totaro's topological result for surfaces.
\end{rem}

\begin{Cor} Let $k=\CC$. Writing $Br(X^{[2]})=(\QQ/\ZZ)^a \oplus T$, we have
\begin{center} 
\label{table}
\begin{tabular} { |c|c|c| }
\hline 
Surface & $a$ & Torsion part \\
\hline
$K3$  & $22 -NS(X)$ & $0$ \\
\hline
Abelian surface & $22-2NS(X)$ & 0 \\
\hline
$\PP^2_\CC$ & $0$ & $0$ \\
\hline
$\FF_n$ & $0$ & $0$ \\
\hline
del Pezzo & $0$ & $0$ \\
\hline
Barlow  \cite{barlow} & $0$ & $0$ \\
\hline
\cite[Lee, Park]{leepark} & $0$ & $0$ \\
\hline
\cite[Park, Park, Shin]{pps} & $0$ & $0$ \\
\hline
Godeaux & $0$ & $\ZZ/5\ZZ^{\oplus 2}$ \\
\hline
\cite[Catanese]{catanese1981} & 0 & $\ZZ/5\ZZ^{\oplus 2}$ \\
\hline
%$C_1 \times C_2$ & $b_1+b_2$ & 0 \\
%\hline
\end{tabular}
\end{center}
\end{Cor}

\begin{proof}
For a $K3$ surface $X$, the Albanese variety is trivial. So $C_s(X)=NS(Alb(X))$ has rank zero. So $a$  solves to $22-NS(X)$.

Now $\PP^2_\CC$ and $\FF_n$ are simply-connected and have trivial Albanese variety. Further for $\PP^2$ we have $b_2=1$, $b_1=0$, $NS(X)=1$ while for $\FF_n$ we have $b_2=2$, $b_1=0$ and $NS(X)=2$. 

For an abelian surface, since $H^n(X,\ZZ)= \wedge^n H^1(X,\ZZ)$, so the torsion part of the Brauer group is trivial. Now $b_1=4$, so $b_2= {4 \choose 2}=6$. Further the Albanese of $X$ is itself. So we get $b=6+16-NS(X)-NS(X)=22-2NS(X)$.

 Since a del Pezzo surface is birational to $\PP^2$, so it is simply-connected. Hence the torsion part of the Brauer group is trivial and $H_1(X)=0$. Thus $b_1=0$. Further the Albanese variety is trivial. Finally $b_2=NS(X)$ since $H^i(X,\cO_X)=0$ for $i>0$. Hence we get $b_2+b_1^2- rank NS(X)-rank C_s(X)=0$.

%Products of curves also has torsion-free singular cohomology.

Barlow \cite{barlow}, \cite[Lee,Park]{leepark} and \cite[Park,Park,Shin]{pps} have constructed examples of simply-connected surfaces of general type with $p_g=0$. It is easy to see that for such surfaces $X$, we have $Br(X^{[2]})=0$.

A Godeaux surface is the quotient of the Fermat surface by the cyclic group of order $5$. For a Godeaux surface or Catanese surface, we have $q=p_g=0$. Thus $Alb(X)$ is trivial and $b_2-\rank NS(X)=0$. Further their fundamental group is $\ZZ/5\ZZ$. Thus $b_1=0$. So the divisible subgroup of the  Brauer group is trivial.  Therefore we get $H^3(X,\ZZ)_{tor}$ which by Poincar\'e duality equals $H_1(X)_{tor}=\ZZ/5\ZZ$. Also $ Tor_1^{\ZZ}(H^2(X,\ZZ),H^2(X,\ZZ))=\ZZ/5$.
\end{proof}

%\begin{Cor} Let $X$ be a Godeaux or Catanese surface. Then $X^{[2]}$ is not rational.\end{Cor}\begin{proof} We have shown that $Br(X^{[2]})=\ZZ/5^{\oplus 2}$ in these cases. Now the result follows from \cite[Prop 1]{artinmumford}.\end{proof}

\section{Case $d \geq 3$} \label{dgeq3} In this section, $char(k) \neq 2$. Further, by a series of reductions we show how the
case $d=2$ generalizes to $d \geq 3$. Our first reduction is from $X^{[d]}$ to a certain open subset $X^{[d]}_*$.

Recall (\ref{hilbtochow}) the Hilbert to Chow morphism  \begin{equation} \label{hilbertchow} \omega_d: X^{[d]} \ra X^{(d)}\end{equation}defined by $Z \ms \sum_{p \in Supp(Z)} l(\cO_{Z,p}) [p]$ where $p$ is a closed point of $X$. The scheme $X^{(d)}$ admits a stratification by locally closed subschemes parametrized by unordered partitions $<n_1,\cdots,n_r>$ where $r \geq 1$ and $\sum n_i =d$. Since $dim(X)=2$, so the dimension of such locus of points is $2r$ and over them the dimension of the fiber of $\omega_n$ is $d-r$ by \cite{iarrobino}. Let $W_1 \subset X^{(d)}$ denote the open subset consisting of all distinct points i.e $<1,\cdots,1>$ repeated $d$-times and $W_2 \hra X^{(d)}$ be the locally closed subscheme consisting of exactly two identical points i.e $<2,1,\cdots,1>$ where there are $d-2$ ones. Therefore $dim W_1= dim X^{(d)}$ is $2d$ and $dim W_2$ is $2d-2$ and the remaining stratas have strictly smaller dimensions. 

Set \begin{eqnarray} \label{technical} 
X^{(d)}_* & := & W_1 \cup W_2 \\
X^{[d]}_* & := & \omega_d^{-1}(X^{(d)}_*) \\
V_1       & := & \omega_d^{-1}(W_1) \\
\label{v2} V_2 & := & \omega_d^{-1}(W_2).
\end{eqnarray}  
 
 We remark firstly that $X^{[d]}$ is smooth (cf \cite{fogarty}).  The dimension of the inverse image under $\omega_d$ of any strata given by partition $<n_1,\cdots,n_r>$ equals $2r+d-r$. So its codimension equals
 \begin{equation} 2d -(2r+d-r)=d-r.
 \end{equation} Further for the inverse image of  any strata other than $W_1$ and $W_2$, we have $d-r \geq 2$. Thus the complement of $X^{[d]}_*$ in $X^{[d]}$ has codimension two. It is well-known that the natural restriction map for
 \begin{equation} 
 Br(X^{[d]}) \ra Br(X^{[d]}_*)
 \end{equation}
  is an isomorphism.

Consider the natural morphism\begin{equation} \omega_d: X^{[d]}_* \ra X^{(d)}_*. \end{equation} We now quote three facts that will be useful in analysing the morphism $\omega_d$. Let \begin{equation} \label{singlocusgen} X^{(d)}_s\end{equation} denote the singular locus of $X^{(d)}$. We have \begin{equation} \label{singloc} X^{(d)}_* \cap X^{(d)}_s=W_2.\end{equation}By \cite[Lemma 4.4]{fogarty}  the blowup of $X^{(d)}_*$ along its singular locus is $X^{[d]}_*$ i.e\begin{equation} \label{blowup}X^{[d]}_* = Bl_{W_2} X^{(d)}_*.\end{equation} When $char(k) \neq 2$, for any $q \in W_2$ the schematic fiber of $\omega_n^{-1}(q)$ is the reduced scheme $\PP^1_{k}$ by \cite[ \S 3.3 and Prop 3.3.3]{ps} (cf also \cite[page 668 Lemma 4.3]{fogarty}).

Before we proceed with the actual proof, we wish to reinterpret the scheme $V_2$ as the projectivization of a bundle on $W_2$. 

\begin{prop} The scheme $V_2 \ra W_2$ is the projectivization of a rank two vector bundle on $W_2$.
\end{prop}
\begin{proof} Consider the locally closed subscheme $L$ of $X^{\times d}$ where exactly two coordinates are equal. The scheme $L$ has components parametrized by unordered pairs $\{(i,j)| 1 \leq i \neq j \leq n\}$ i.e $(i,j)  =(j,i)$. The component $X^{\times d}_{ij}$ of $L$ has equal $i$th and $j$th coordinates. By the definition of $L$, the intersection of any two components is empty. There is a natural  action of the symmetric group $S_d$ on $d$ letters on $L$ that respects the inclusion $L \hra X^{\times d}$. It acts on the set of components transitively. Further the following diagram is commutative
\begin{equation}
\xymatrix{
L \ar@{^{(}->}[r] \ar[d] & X^{\times d} \ar[d] \\
W_2 \ar@{^{(}->}[r] & X^{(d)}
}
\end{equation}
where the vertical morphisms are quotients under $S_d$-action. Here the component $X^{\times d}_{ij}$ maps isomorphically onto $W_2$.

Let $p_{ij}: X^{\times d}_{ij} \ra X$ denote the natural projection map onto the $i$-th (or the $j$-th) factor. These glue to give a  natural map $\pi': L \ra X$ which is $S_d$ invariant. Thus we get a projection map 
\begin{equation}
\pi: W_2 \ra X.
\end{equation}
Consider the rank two bundle $W_2 \times_X \Omega_X$ on $W_2$. Recall $V_2$ from (\ref{v2}). We claim that we have a natural isomorphism of $\PP^1$ bundles on $W_2$
\begin{equation}
V_2=\PP(W_2 \times_X \Omega_X).
\end{equation} 
To see this, let us first define a natural morphism
\begin{equation} \theta: V_2 \ra X^{[2]},
\end{equation} as follows: a point $v \in V_2 \subset X^{[d]}$ corresponds to an ideal subsheaf $I$ of $\cO_X$. By the definition of $V_2$, this can be written as the ideal genearted by ideals
\begin{equation} I=(m_1, \cdots, m_{d-2}, I')
\end{equation} where the $m_i$ are maximal ideals corresponing to the distinct points of multiplicity one and $I'$ is an ideal subsheaf of $\cO_X$ of colength two. Define $\theta$ as the morphism that sends $v$ to the point 
\begin{equation} \cO_X/I' \in X^{[2]}.
\end{equation}
So we get a factorization
\begin{equation} \label{parmudiag}
\xymatrix{
X^{[2]} \ar[d] & \PP(\Omega_X) \ar[d] \ar@{_{(}->}[l] & V_2 \ar[d] \ar@{.>}[l] \ar[r] \ar@/_1pc/[ll] & X^{[d]}_* \ar[d] \\
X^{(2)} & X \ar@{_{(}->}[l] & W_2 \ar[l] \ar[r] & X^{(d)}_*
}
\end{equation}
  This shows the claim and the proposition.
\end{proof}
We remark that all squares in (\ref{parmudiag}) are cartesian.
The following proposition reduces the case $d \geq 3$ to $d=2$.

\begin{thm} \label{geq3} Let $char(k) \neq 2$. The map $\omega_d: X^{[d]}_* \ra X^{(d)}_*$ induces an isomorphism on cohomological Brauer groups.
\end{thm}
\begin{proof} The diagram (\ref{parmudiag}) plays the role of 
  the diagram (\ref{diag2}) i.e $V_2=\PP(W_2 \times_X \Omega_X)$, $X^{[d]}_*$, $X^{(d)}_*$ and $W_2$ play the roles of $\PP(\Omega_X)$, $X^{[2]}$, $X^{(2)}$ and $X$ respectively. Further $S_d$ plays the role of $S_2$. Now  Proposition \ref{A1}  holds for any $d \geq 2$.  For any $d$ the copy of $\ZZ$ in (\ref{picxd}) comes from the divisor corresponding to $\omega_d^{-1}(X^{(d)}_*)$, which is $V_2$ in our  notation (cf \cite[Theorem 6.2]{fogarty}).  Alongwith  this last fact, the only properties we used in 
the proofs of Propositions  \ref{pushforwardbyomega2}, \ref{onemoreh2} and \ref{H_3tor} were that the fibers are points or the projective line. So this shows that
 the proofs of Theorem \ref{anychar} and \ref{mtcomplex} generalize. 
\end{proof}

\section{Cohomological Brauer group of punctual Quot scheme $Q(r,d)$} As in the introduction, let $Q(r,d)$ denote the punctual Quot scheme parametrizing torsion quotients of $\cO_X^{\oplus r}$ of degree $d$ supported on zero dimensional subschemes. We have the diagonal action of $\GG_m^r$ on $\cO_X^{\oplus r}$. Since $Q(r,d)$ is projective,   we have the Bialynicki-Birula decomposition of $Q(r,d)$ with respect to the diagonal embedding $\GG_m \ra \GG_m^r$ for the above action. Although $Q(r,d)$ may not be smooth, the fixed point locus and the stratas may be described as in \cite[Thm 1]{bifet} as follows. We begin with the description of fixed points. The connected components of fixed points  are parametrized  by partitions $\ul{P}=d_1 + \cdots + d_r$  of $d$ into $r$-many non-negative integers $d_i$. Let us set
\begin{equation}  \label{fixedpoints}
X^{\ul{P}} := \prod_{i=1, \cdots r} X^{[d_i]},
\end{equation}
as a fibered product of Hilbert schemes with the convention that $X^{[0]}=\Spec(k)$.
The strata $S^{\ul{P}}$  corresponding to the partition $\ul{P}$ is a vector bundle over $X^{\ul{P}}$ 
\begin{equation} \rho: S^{\ul{P}} \ra X^{\ul{P}}
\end{equation}  of rank $d_{\ul{P}}= \sum_{i <j} d_{ij}( \ul{I})$ where $\ul{I}=(I_1,\cdots,I_r)$ is any closed point of $X^{\ul{P}}$ and
\begin{equation}
d_{ij}(\ul{I})=dim_k Hom_{\cO_X}(I_i,\cO_X/I_j).
\end{equation} 
Let $\GG_m=\Spec k[t,t^{-1}]$ and $s \in S^{\ul{P}}$. Under a $\GG_m$-action given by $\ul{P}$ the morphism $\rho$ on $S^{\ul{P}}$ may be described as 
\begin{equation}
\rho(s)=lim_{t \ra 0} t. s.
\end{equation}
 
An analogue of the following proposition is proved in \cite{bifet} for the case when $X$ is a curve. We provide the additional arguments for the case $X$ is a surface because we need to ascertain the unique strata of largest dimension corresponds to the partition $\ul{P}:=(0,\cdots,0,d)$ of $d$ into $r$ non-negative parts.

 \begin{prop} Let $\ul{P}=(d_1,\cdots,d_r)$ be a partition where $d_1 \leq d_2 \leq \cdots \leq d_r$. Then the dimension of $S^{\ul{P}}$ is $\sum_{j=1}^r (j-1) d_j$. 
 \end{prop}
 \begin{proof} Since we may choose any point $\ul{I}$ in $X^{\ul{P}}$, let us choose a point where each $I_i$ corresponds to $d_i$ many distinct points and further the supports of $I_i$ are also all distinct. Let the subscheme $I_i$ correspond to the sheaf of ideals $J_i \subset \cO_X$. So we have the exact sequence
 \begin{equation}
 0 \ra J_i \ra \cO_X \ra \cO_{I_i} \ra 0.
 \end{equation}
 Now applying $Hom_{\cO_X}(?,\cO_{I_j})$ we get
 \begin{equation*}
 0 \ra Hom(\cO_{I_i},\cO_{I_j}) \ra Hom(\cO_X,\cO_{I_j}) \ra Hom(J_i,\cO_{I_j}) \ra Ext^1_{\cO_X}(\cO_{I_i},\cO_{I_j}) \ra
 \end{equation*}
 Now $Hom(\cO_{I_i},\cO_{I_j})=0$ and $Ext^1_{\cO_X}(\cO_{I_i},\cO_{I_j})=0$ because the supports of $I_i$ and $I_j$ are distinct. Thus $Hom(\cO_X,\cO_{I_j}) \ra Hom(J_i,\cO_{I_j})$ is an isomorphism. This shows that $d_{ij}(\ul{I})=dim Hom(J_i,\cO_{I_j})=d_j$ since $\cO_{I_j}$ has length $d_j$. Now the result follows from $d_{\ul{P}}= \sum_{i <j} d_{ij}( \ul{I})$.
  
 \end{proof}
 Considering dimensions by (\ref{fixedpoints}) there is a  unique strata of largest dimension. It is thus open and given by 
 \begin{equation}
 X^{\ul{D}}=X^{[0]} \times X^{[0]} \times \cdots \times X^{[d]}= X^{[d]}.
 \end{equation}
 
 The following proposition is inspired from \cite[Remark 6.3]{bidhhu}.
 
 \begin{thm} The morphism $\phi: Q(r,d) \ra X^{[d]}$ (cf \ref{quothilb}) induces an isomorphism on cohomological Brauer groups.
 \end{thm}
  \begin{proof}
 %It has dimension $(r-1)d$. The next biggest strata is also unique and corresponds to $(0,\cdots,0,1,d-1)$. It has dimension $(r-2) \times 1 + (r-1)(d-1)$.  Therefore the complement of the biggest strata has codimension two.
 We have a Zariski local fibration $\rho: S^{\ul{P}} \ra X^{\ul{P}}=X^{[d]}$ with fibers isomorphic to affine spaces.  So it induces isomorphism on cohomological Brauer groups. 
 Let $s: X^{\ul{D}} \ra S^{\ul{D}}$ denote the zero section. It induces an isomorphism on cohomological Brauer groups because $\rho$ does. It sends a quotient $q: \cO_X \ra S$ to the quotient \begin{equation} (0, \cdots,0,q): \cO_X^{r-1} \oplus \cO_{X} \ra S 
 \end{equation}
in $S^{\ul{P}}$. Set $\ul{D}=(0,\cdots,0,d)$. Consider the diagram
 \begin{equation} \label{referee}
 \xymatrix{
 S^{\ul{D}} \ar@{^{(}->}[r]^{i}  & Q(r,d) \ar[ld]^{\phi} \\
X^{[d]}= X^{\ul{D}} \ar[u]^s&
 } 
 \end{equation}
 Thus $i \circ s$ is a section of  $\phi$. Hence $(i \circ s)^* \circ \phi^*=id_{Br'(X^{[d]})}$ in the diagram below
 \begin{equation}
 \xymatrix{
 Br'(S^{\ul{D}})  \ar[d]^{s^*} & Br'(Q(r,d)) \ar@{_{(}->}[l]_{i^*}  \\
 Br'(X^{[d]})  \ar[ur]_{\phi^*} &
 }
 \end{equation}
 Thus $s^* \circ i^*$ is surjective. So $i^*$ is surjective since $s^*$ is an isomorphism. But it is also
 injective because $i$ is the inclusion of a non-empty open set (cf \cite[IV, Corollary 2.6]{milne}). Thus $i^*$ is an isomorphism.  Now $\phi^*$ equals $ (i^*)^{-1}(s^*)^{-1}$,so it is also an isomorphism.
 
%Take $b \in Br'(Q(r,d))$. Consider $\phi^* s^*  i^* (b)$. Since $s^* i^* \phi^* s^* i^* (b)=s^* i^* (b)$, and $s^*,i^*$ are isomorphisms, so it follows that $b=\phi^* s^* i^* (b)$ i.e $\phi^*$ is also surjective.
 \end{proof}

\bibliographystyle{plain}
\bibliography{brauerquot}
 \end{document}